\documentclass[11pt]{elsarticle}
\usepackage{amssymb}
\usepackage{amsmath}
\usepackage{amsfonts,graphicx,color, bbm}
\usepackage{mathtools}
\usepackage{tikz}
\usepackage{graphicx}
\usepackage{subcaption}

\newtheorem{theorem}{Theorem}[section]
\newtheorem{lemma}{Lemma}[section]
\newtheorem{corollary}{Corollary}[section]

\newtheorem{remark}{Remark}[section]
\newcommand{\ignore}[1]{}{}
\newproof{pf}{Proof}
\newproof{pot1}{Proof of Theorem \ref{Thm1}}
\newproof{pot2}{Proof of Corollary \ref{Cor1}}
\newproof{pot3}{Proof of Theorem \ref{Thm2}}
\newproof{pot4}{Proof of Theorem \ref{Thm3}}
\newtheorem{condition}{Condition}

\newcommand\beq{\begin{equation}}
	\newcommand\eeq{\end{equation}}

\makeatletter

\newcommand{\Rmnum}[1]{\expandafter\@slowromancap\romannumeral #1@}

\def\geq{\geqslant}
\def\leq{\leqslant}
\def\1{\mathbbm{1}}
\def\l{\lambda}
\def\C{\mathbb{C}}
\def\P{\mathbb{P}}
\def\E{\mathbb{E}}
\def\N{\mathbb{N}}
\def\R{\mathbb{R}}
\def\X{\mathbb{X}}
\def\w{\widetilde}
\def\Z{\mathbb{Z}}
\def\bfP{\mathbf{P}}
\def\bfn{\boldsymbol{\nu}}

\begin{document}
	\let\today\relax
\date{\quad}
\begin{frontmatter}



\title{Limit theorems for supercritical branching processes in Markovian environment with linear fractional offspring distributions}
\author[cor1]{Yinna Ye}
\author[cor2]{Quansheng Liu
}

\address[cor1]{Department of Applied Mathematics, School of Mathematics and Physics, Xi'an Jiaotong-Liverpool University, Suzhou 215123, China}	
\address[cor2]{LMBA, UMR CNRS 6205, Universit\'{e} de Bretagne-Sud, Vannes
	 56000, France}	


\begin{abstract}
In this paper, we study the asymptotic behaviour for distribution of supercritical branching process in a Markovian environment with linear fractional offspring distributions. Phase transitions in the behaviour of the process is discovered, through investigating the strongly, intermediately and weakly supercritical cases.  
\end{abstract}




\end{frontmatter}
\date{\today}

\section{Introduction and main results}
\setcounter{equation}{0}
Branching process in random environment (BPRE) models population evolution where individuals reproduce independently and the offspring distribution fluctuates randomly across generations. The limit theorems for BPRE, especially for branching processes in independent and identically distributed (i.i.d.) environment or stationary and ergodic environment, have been studied widely and intensively since the 1970s. For instance, one may refer to \citet{AK71b,AK71a}, \citet{T77,T88}, \citet{GL01} and \citet{BGK05} for some classical results.

Initially introduced by \citet{AK71b} in 1971, branching process in Markovian environment (BPME) is a generalization of the classical BPRE with i.i.d. or stationary and ergodic environment. In their paper \cite{AK71b}, they found sufficient conditions for the certainty of ultimate exinction (Theorem 4), in the case when the environment process is an irreducible and positive recurrent Markov chain with countable state space. Only until the past decade, there has been some advances in the study of BPME. \citet{LY10} introduced a BPME with environment process driven by a bivariate Markov chain. They established the asymptotic behaviour for survival probability in the critical regime. \citet{GLL19,GLL22} determined the asymptotic behaviours of survival probability, in both critical and subcritical regimes, for BPME with environment process driven by a univariate Markov chain. While, the limit theorems for supercritical BPME have still rarely appeared so far in the literature.

The main objective of this article is to determine the exact asymptotic behaviours for probability distribution of supercritical BPMEs. To this end, we studied the phase transitions among strongly, intermediately and weakly supercritical BPME with linear fractional offspring distributions in the same settings as \citet{GLL19}, i.e. the environment process is driven by a univariate Markov chain; and characterized these regimes with limit theorems. 

Let us start with introduction to linear fractional distributions. A probability distribution (p.m.f) $q=\left(q(k)\right)_{k\geq0}$ on $\N_0:=\N\cup\{0\}$ is called \textit{linear fractional} (l.f.), if it satisfies
\begin{equation}\label{def-l.f.}
	q(0)=a\quad\text{and}\quad q{(k)}=(1-a)(1-p)p^{k-1},\quad\text{for }k\in\N;
\end{equation}
 where $a\in[0,1)$, $p\in(0,1)$ and are the parameters of the l.f. distribution. And its probability generating function (p.g.f) is given by
 $$f(s):=1-\frac{(1-a)(1-s)}{1-ps},\quad s\in[0,1].$$
 Moreover, the expectation and normalized second factorial moment of l.f. distribution are respectively given by
 \begin{equation}\label{basic}
 	m:=f'(1)=\frac{1-a}{1-p}\quad\text{and}\quad \eta:=\frac{f''(1)}{2f'(1)^2}=\frac{p}{1-a}.
 \end{equation}

Now, Let us describe the model of BPME with l.f. offspring distributions. Consider in a particle system, each individual reproduces from generation to generation some random number of children following a random l.f. offspring distribution, which depends on the current state of the random environment. More precisely, consider a time-homogeneous Markov chain $X=(X_n)_{n\geq0}$ defined on the probability space $(\Omega, \mathcal{F}, \P)$, with finite and countable state space $\X$. The Markov chain $X$ represents the random environment in this model. For $i\in\X$, let $\P_i$ be the probability law on $(\Omega, \mathcal{F})$ generated by the finite dimensional distributions of
$(X_n)_{n\geq0}$, starting at $X_0=i$. Denote by $\E$ and $\E_i$ the expectations
associated to $\P$ and $\P_i$, respectively. Let $\mathcal{C}$ be the set of functions mapping from $\X$ to $\C$. Denote by $\bfP$ the Markov kernel of the chain $X$, given by
$$\bfP g(i):=\E_i\left[g(X_1)\right],$$
for any $g\in\mathcal{C}$ and $i\in\X$. Let $Z_n$ denote the total population of the $n$th generation in the particle system. Then BPME $(Z_n)_{n\geq0}$ is defined by the following recursive relation:
\begin{equation}\label{recursive}
Z_0=1\quad\text{and}\quad Z_n=\sum_{j=0}^{Z_{n-1}}\xi_{X_n}^{n,j},\quad n\in\N,
\end{equation}
where given $X_n=i$, the random variable $\xi^{n,j}_i$ represents the number of children produced by the $j$th individual in the $(n-1)$th generation, for $j\in\{1,2,\ldots, Z_{n-1}\}$. Assume that $(\xi^{n,j}_i)_{j,\,n\geq1}$ is an i.i.d. sequence with the common p.m.f $q_i$ and p.g.f $f_i$ defined on the probability space $(\Omega, \mathcal{F},\P)$ and given by, for any $i\in\X$,
$$f_i(s):=1-\frac{(1-a_i)(1-s)}{1-p_i\,s},\quad s\in[0,1];$$
and $(a_i,p_i)$ characterizes a l.f. distribution in the same manner as (\ref{def-l.f.}) with $a_i\in[0,1)$ and $p_i\in(0,1)$. Assume also that the sequence $(\xi^{n,j}_i)_{j,\,n\geq1}$ is independent with the Markov chain $X$. 

\begin{remark}
	In particular, when the state space $\X=\{i_0\}$ is a singleton, then $\displaystyle(\xi^{n,j}_{i_0})_{j,\,n\geq1}$ are i.i.d. with the common p.g.f $f_{i_0}$. In this case, our model becomes a Galton-Watson process with the l.f. offspring distribution characterized by the parameters $(a_{i_0},p_{i_0})$.
\end{remark}

For any $ i\in\X$, let $\rho(i):=\ln f_i'(1)$. Consider the Markov walk $(S_n)_{n\geq0}$ associated with the branching process, defined by 
$$S_0:=0\quad\text{and}\quad S_n:=\ln\left[\prod_{i=1}^{n}f'_{X_i}(1)\right]=\sum_{i=1}^n\rho(X_i),\quad n\in\N.$$
Throughout the paper, assume the following conditions.
\begin{condition} [Irreducible and aperiodic environment]\label{Con1}
	The matrix $\bfP$ is primitive, i.e. there exists constant $k_0\in\N$ such that for any non-negative and non-identically zero function $g\in\mathcal{C}$ and $i\in\X$, it holds $\bfP^{k_0} {g(i)>0}$.
\end{condition}
This condition is equivalent to the one that the Markov chain $X$ is irreducible and aperiodic. From Perron-Frobenius theorem, $\bfP$ has a unique positive invariant probability $\bfn$ on $\X$ and the following convergence holds, for any $(i,j)\in\X^2$,
$$\lim_{n\rightarrow+\infty}\bfP^n(i,j)=\bfn(j).$$ 

\begin{condition}[Non-lattice environment]\label{Con2}
	For any $(\theta, b)\in\R^2$, there exists a path $x_0, \ldots, x_n$ in $\X$ such that 
	$$\bfP(x_0,x_1)\cdots \bfP(x_{n-1},x_n)\,\bfP(x_n,x_0)>0$$ 
	and 
	$$\rho(x_0)+\cdots+\rho(x_n)-(n+1)\theta\notin b\Z.$$
\end{condition}
This condition means that the Markov walk $(S_n)_{n\geq0}$ is non-lattice.

For any $\l\in\R$ and $i\in\X$, let
$$k(\l):=\lim_{n\rightarrow+\infty}\E^{1/n}_i\left(e^{\l S_n}\right).$$
According to \citet{GLL19} (Section 2.4), under Conditions \ref{Con1} and \ref{Con2}, the above limit exists and does not depend on the initial state of the chain $X_0=i$. Moreover, the BPME $(Z_n)_{n\geq0}$ is called supercritical, critical or subcritical according to $k'(0)>0$, $k'(0)=0$ or $k'(0)<0$, respectively. Indeed, for a given $\l\in\R$, $k(\l)$ is related to the transfer operator $\bfP_\l$ defined by
$$\bfP_\l\, g(i):=\bfP\left(e^{\l\rho}g\right)(i)=\E_i\left[e^{\l S_1}g(X_1)\right],$$
for any $g\in\mathcal{C}$ and $i\in\X$. And from Perron-Frobenius theorem, there exists positive function $\nu_\l$ on $\X$ such that $k(\l)$ is an eigenvalue associated to the eigenvector $\nu_\l$ for $\bfP_\l$. As a result, after normalization as $\displaystyle\w{\bfP}_\l g:=\frac{\bfP_\l(g\nu_\l)}{k(\l)\,\nu_\l}$, such a new transfer operator $\w{\bfP}_\l$ becomes Markovian and then also becomes a Markov kernel. There are more detailed properties about the operators $\bfP_\l$ and $\w{\bfP}_\l$ in Section \ref{Sec2.1}.

\begin{condition}[Supercritical branching]\label{Con3}
	For any $i\in\X$, the parameters $\left(a_i,p_i\right)$ of the l.f. offspring distributions satisfy 
	\begin{equation}\label{supercritical}
		0<a_i<p_i<1.
	\end{equation}
\end{condition}
The condition (\ref{supercritical}) implies that for any $i\in\X$ and $k\in\N_0$, $0<q_i(k)<1$. Consequently, by (\ref{basic}), we can see that the expectation $m_i$ and variance $\sigma_i^2$ of the l.f. offspring distribution satisfy respectively
$$m_i:={f_i}'(1)\in(0,+\infty)\quad\text{and}\quad\sigma_i^2:={f_i}''(1)-\left({f_i}'(1)\right)^2\in(0,+\infty),$$
for any $i\in\X$. Moreover, from (\ref{basic}), we have (\ref{supercritical}) implies that $m_i>1$, for any $i\in\X$. As a result, by Lemma 2.15 of \cite{GLL19}, we have $k'(0)=\bfn(0)>0$; i.e. BPME $\{Z_n\}_{n\geq0}$ is supercritical, if  (\ref{supercritical}) is satisfied.

In this paper, the research effort will focus on the supercritical case, that is when $k'(0)>0$. We say that the BPME is \textit{strongly supercritical} if $k'(0)>0$, $k'(-1)>0$, \textit{intermediately supercritical} if $k'(0)>0$, $k'(-1)=0$ and \textit{weakly supercritical} if $k'(0)>0$, $k'(-1)<0$. By Lemma 2.15 of \cite{GLL19}, 
\begin{equation}\label{coherence}
	k'(-1)/k(-1)=\w{\bfn}_{-1}(\rho)=\E_{\w{\bfn}_{-1}}\left[\rho(X_1)\right],
\end{equation}
where $\E_{\w{\bfn}_{-1}}$ is the expectation generated by the finite dimensional distributions of the Markov chain with Markov kernel $\w{\bfP}_{-1}$ in the stationary regime. See Figure \ref{figure} below for different curve shapes of the function $k$ in the three sub-cases. When $(X_n)_{n\geq1}$ is an i.i.d. sequence with common law $\bfn$, then the model is reduced to branching process in i.i.d. environment. In this case, in addition to (\ref{coherence}), we have
\begin{equation}\label{coherence1}
	\E_{\w{\bfn}_{-1}}\left[\rho(X_1)\right]=\E_{\bfn}\left[\rho(X_1)\,e^{-S_1}\right]=\E_{\bfn}\left[\rho(X_1)\,e^{-\rho(X_1)}\right]=\E_{\bfn}\left[\frac{\ln f'_{X_1}(1)}{f'_{X_1}(1)}\right].
\end{equation}
Since $k(\l)>0$, for any $\l\in\R$, we can see from (\ref{coherence}) and (\ref{coherence1}) that for supercritical branching processes in i.i.d. environments, our way of classification to strongly, intermediately and weakly supercritical regimes using $k'(-1)$ is equivalent to the one using $\E_{\bfn}\left[\ln f'_{X_1}(1)/f'_{X_1}(1)\right]$ by \citet{B14}, i.e. strongly, intermediately or weakly supercritical regimes according to $\E_{\bfn}\left[\ln f'_{X_1}(1)/f'_{X_1}(1)\right]>0$, $=0$ or $<0$, respectively.  

\begin{figure}[htbp]
	\centering
	\begin{subfigure}[b]{0.32\textwidth}
		\centering
		\includegraphics[width=\textwidth]{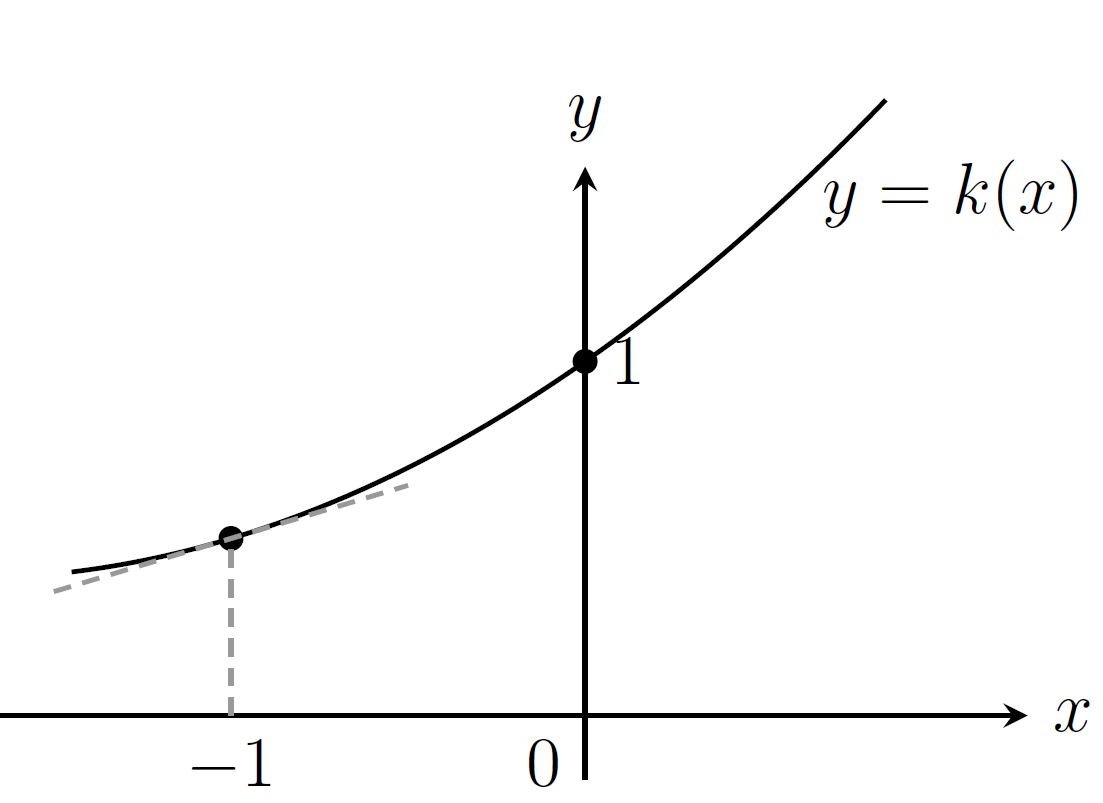}
		\caption{Strongly supercritical case: $k'(-1)>0$.}
		\label{figure-strong}
	\end{subfigure}
	\hfill
	\begin{subfigure}[b]{0.32\textwidth}
		\centering
		\includegraphics[width=\textwidth]{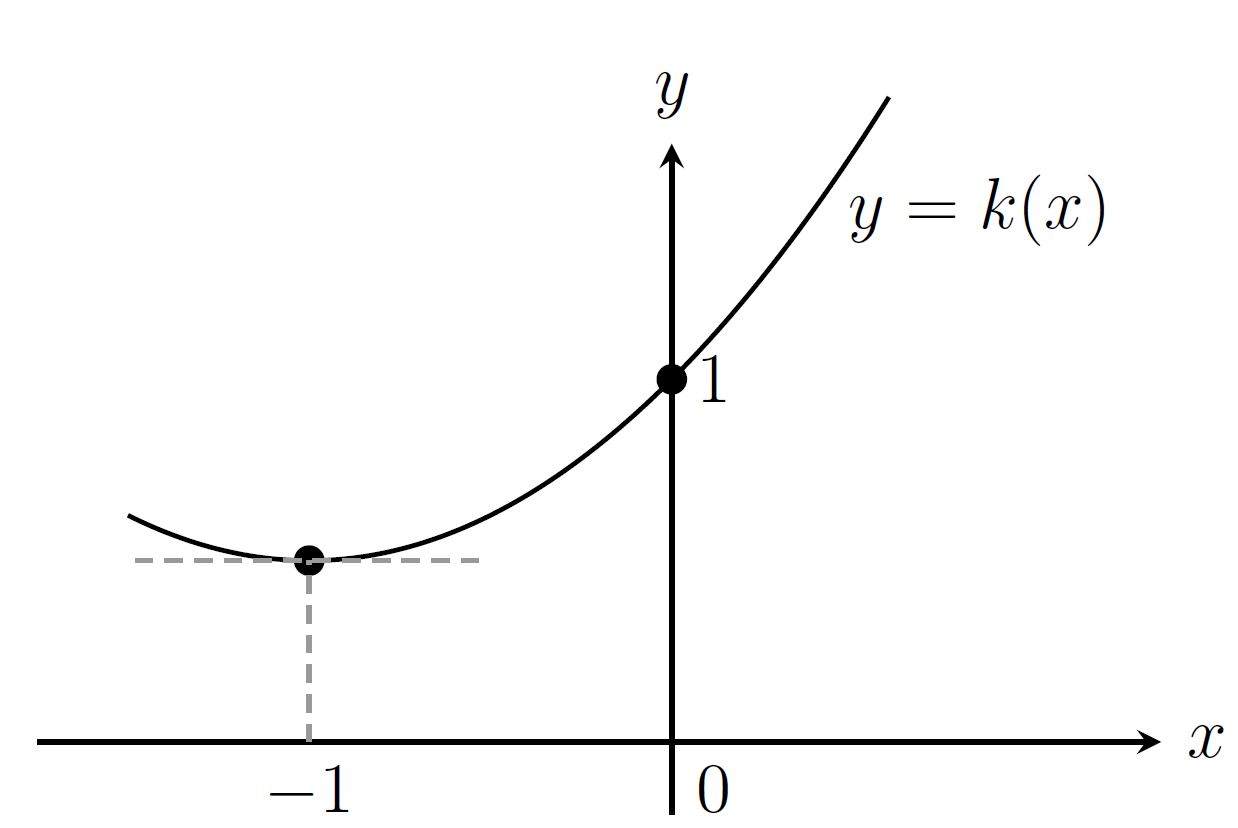}
		\caption{Intermediately supercritical case: $k'(-1)=0$.}
		\label{figure-intermediate}
	\end{subfigure}
	\hfill
	\begin{subfigure}[b]{0.33\textwidth}
		\centering
		\includegraphics[width=\textwidth]{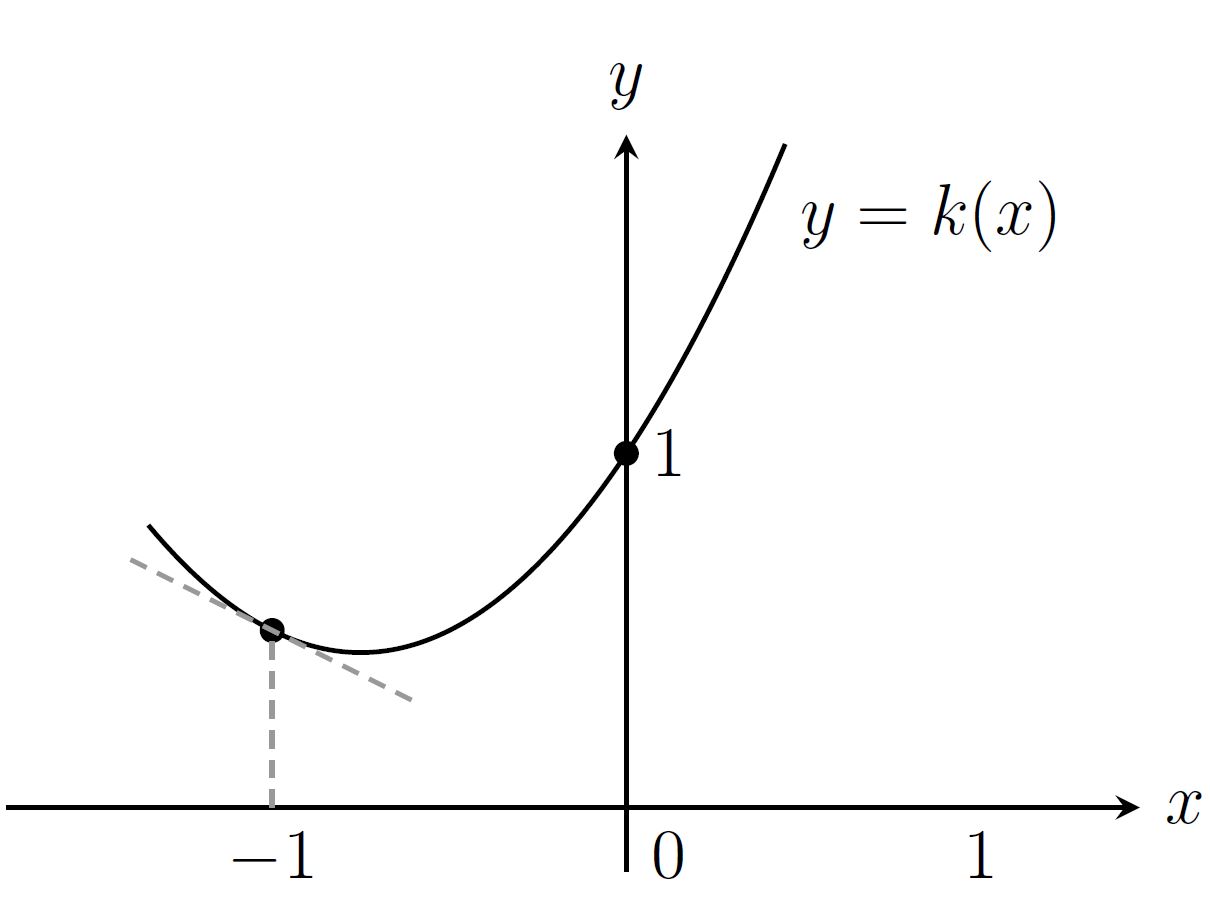}
		\caption{Weakly supercritical case: $k'(-1)<0$.}
		\label{figure-weak}
	\end{subfigure}
	\caption{Curves of strictly convex function $k:\mathbb{R}\to\mathbb{R}$
		satisfying $k(0)=1$ and $k'(0)>0$, with different signs
		of the derivative at $x=-1$.}
     \label{figure}
\end{figure}

We have the following main results for strongly, intermediately and weakly supercritical BPME $(Z_n)_{n\geq0}$ with l.f. distributions, respectively.
\begin{theorem}[Strongly supercritical case] \label{Thm1}
Assume Conditions \ref{Con1}-\ref{Con3} and $k'(-1)>0$. Then there exist positive functions, $u_0(i)$ and $v(j)$, both defined on $\X$, such that for any $(i,j)\in\X^2$ and $z\in\N$, as $n \rightarrow+\infty$,
	$$\P_i(Z_n=z,X_n=j)\sim k^n(-1)\,u_0(i)\,v(j).$$
\end{theorem}

\begin{corollary}\label{Cor1}
Assume that the conditions of Theorem \ref{Thm1} are satisfied. Then for every $c\in\N$, $z\in\{1,2\ldots, c\}$ and $(i,j)\in\X^2$,
$$\lim_{n\rightarrow+\infty}\P_i\left(Z_n=z\,|\,1\leq Z_n\leq c\,,\,X_n=j\right)=\frac{1}{c},$$
i.e. the limit distribution is uniform on $\{1,\ldots,c\}$. 
\end{corollary}

\begin{theorem}[Intermediately supercritical case] \label{Thm2}
	Assume Conditions \ref{Con1}-\ref{Con3} and $k'(-1)=0$. Then there exist positive functions, $u(i)$ and $v(j)$, both defined on $\X$, such that for any $(i,j)\in\X^2$, as $n \rightarrow+\infty$,
	$$\P_i(Z_n=1,X_n=j)\sim \frac{k^n(-1)\,u(i)\,v(j)}{\sqrt{n}}.$$
\end{theorem}

\begin{theorem}[Weakly supercritical case] \label{Thm3}
	Assume Conditions \ref{Con1}-\ref{Con3} and $k'(-1)<0$. Then there exist a unique real value $\l\in(-1,0)$ satisfying $k'(\l)=0$ and $0<k(\l)<1$; and a positive function $u(i,j)$, defined on $\X^2$, such that for any $(i,j)\in\X^2$, as $n \rightarrow+\infty$,
	$$\P_i(Z_n=1,X_n=j)\sim \frac{k^n(\l)\,u(i,j)}{n^{3/2}}.$$
\end{theorem}

 We can see that in the i.i.d. environment case with common law $\bfn$, Theorems 2.1.1, 2.1.2 and 2.2.1 in \cite{B14}, with therein $\gamma:=\E_{\bfn}\left(e^{-\rho(X_1)}\right)=k(-1)$, are respectively analogues to Theorem \ref{Thm1}, Corollary \ref{Cor1} and Theorem \ref{Thm2}.

To obtain the main results above, we investigated the Markov walk $(S_n)_{n\geq0}$ under the transformed probability measures introduced by (\ref{eq-change}). More precisely, in both strongly and intermediately supercritical cases, we changed the original probability measure to a new one $\w{\P}_{-1,i}$ as in (\ref{eq-change}), with $\lambda=-1$. In the strongly supercritical case, $(S_n)_{n\geq0}$ is a Markov walk with positive drift under the new probability measure $\w{\P}_{-1,i}$. Whilst, in the intermediately supercritical case, $(S_n)_{n\geq0}$ is a centred Markov walk under $\w{\P}_{-1,i}$ and $\P_i(Z_n=1,X_n=j)$ relies on the Markov walk $(S_n)_{n\geq0}$ conditioned to stay positive. For this case, we used an additional change of probability measure (\ref{def-p-+}) based on the harmonic function $\w{V}_{-1}$. In the weakly supercritical case, in addition to the change of measure (\ref{eq-change}) with a unique $\lambda\in(-1,0)$ such that $k'(\lambda)=0$ (see also Figure (\ref{figure-weak}) for illustration), we studied the dual Markov walk $(S^*_n)_{n\geq0}$ associated with the dual Markov chain $X^*$.

Throughout the paper, $c$ denotes a positive constant whose value may vary from line to line. And $a\vee b:=\max\{a,b\}$,  $a\wedge b:=\min\{a,b\}$, for any $(a,b)\in\R^2$.

The paper is organized as follows. Section \ref{Sec2} provides some preliminary results. The proofs of Theorem \ref{Thm1} and Corollary \ref{Cor1} in the strongly supercritical case are given in Section \ref{sec-strong}. Section \ref{sec-inter} demonstrates Theorem \ref{Thm2} in the intermediately supercritical case. Finally, the proof of Theorem \ref{Thm3} in weakly supercritical case is given in Section \ref{sec-weak}.
\section{Preliminary results on the associated Markov walk}\label{Sec2}
\setcounter{equation}{0}
This section is aimed to provide some preliminary results, including introduction to change of measure technique and some basic properties on the quenched and annealed laws of $(Z_n)_{n\geq0}$.
\subsection{The change of measure related to the Markov walk}\label{Sec2.1}
In this section, we will introduce an exponential change of the probability measure. This change of measure technique plays important role for the proofs of the main results in this article. For the detailed spectral properties of the new operators created below, one may refer to Section 2.4 in \cite{GLL19}. 

For any $\l\in\R$, let $\bfP_\l$ be the transfer operator on $\mathcal{C}$, defined by
\begin{equation}\label{def-g-l}
\bfP_\l g(i):=\bfP\left(e^{\l\rho}g\right)(i)=\E_i\left[e^{\l S_1}g(X_1)\right],
\end{equation}
for any $g\in\mathcal{C}$ and $i\in \X$. Following the same lines (Equations (2.23) -- (2.27)) as in Section 2.4 of \cite{GLL19}, we can obtain that for a given $\l\in\R$, the matrix $\bfP_\l$ satisfies Condition \ref{Con1}. Moreover, by Perron-Frobenius theorem, there exist positive number $k(\l)>0$, positive function $\nu_\l$ on $\X$ and positive linear form $\bfn_\l:\mathcal{C}\rightarrow \C$, such that $\bfn_\l(\nu_\l)=1$ and $k(\l)$ is an eigenvalue associated to the eigenvector $\nu_\l$, i.e. $k(\l)$ and $\nu_\l$ satisfy for any $i\in\X$,
$$\bfP_\l\,\nu_\l(i)=k(\l)\,\nu_\l(i).$$
Note that when $\l=0$, we have $k(0)=1$, $\nu_0(i)=1$ and $\bfn_0(i)=\bfn(i)$, for any $i\in\X$.

For any $\l\in\R$, define a new Markov operator $\w{\bfP}_\l$ by
$$\w{\bfP}_\l g(i):=\frac{\bfP_\l(g\,\nu_\l)(i)}{k(\l)\,\nu_\l(i)},$$
for any $g\in\mathcal{C}$ and $i\in\X$; and from (\ref{def-g-l}), we have
$$\w{\bfP}_\l g(i)=\frac{\bfP\left(e^{\l\rho} g\,\nu_\l\right)(i)}{k(\l)\,\nu_\l(i)}=\frac{\E_{i}\left[e^{\l S_1}g(X_1)\,\nu_\l(X_1)\right]}{k(\l)\nu_\l(i)}.$$
In particular, when $\l=0$, we can see that $\w{\bfP}_0=\bfP_0=\bfP$.
 
\begin{remark}\label{Rem1}
According to Lemma 2.14 of \cite{GLL19} and what follows, we have if Conditions \ref{Con1} and \ref{Con2} are assumed for $\bfP$, then they are also satisfied for the operator $\w{\bfP}_\l$. Furthermore, for $\l\in\R$, as  Markov kernel, $\w{\bfP}_\l$ has a positive invariant measure $\w{\bfn}_\l$ given by
$$\w{\bfn}_\l(g)=\bfn_\l(g\nu_\l),$$
for any $g\in\mathcal{C}$. And for any $(i,j)\in\X^2$, the following convergence holds
\begin{equation}\label{invariant}
	\lim_{n\rightarrow+\infty}\w{\bfP}^n_\l(i,j)=\w{\bfn}_\l(j).
\end{equation}
\end{remark}

For $\l\in\R$ and $i\in\X$, let $\w{\P}_{\l,i}$ and $\w{\E}_{\l,i}$ be the probability and expectation respectively generated by the finite dimensional distributions of the Markov chain $(X_n)_{n\geq0}$ with Markov kernel $\w{\bfP}_\l$ and starting at $X_0=i$. That is,
\begin{equation}\label{eq-change}
	\w{\E}_{\l,i}\left[g(X_1,\ldots,X_n)\right]:=\frac{\E_{i}\left[e^{\l S_n}g(X_1,\ldots,X_n)\,\nu_\l(X_n)\right]}{k^n(\l)\nu_\l(i)},
\end{equation}
for any $n\in\N$, function $g:\X^n\rightarrow \C$ and $i\in\X$. In particular, when $\l=0$, we can see that $\w{\P}_{0,i}$ is reduced to $\P_i$.

\subsection{Quenched and annealed laws of the BPME}\label{Sec2.2}
In this section, we will study the quenched and annealed laws of the BPME $\{Z_n\}_{n\geq0}$, respectively.

Let's start with some notations, which will be used throughout the paper. For any $n\in\N$, $k\in\{1,\ldots,n\}$, $i\in\X$ and $s\in[0,1)$, set 
$$f_{k,n}(s):=f_{X_k}\circ\cdots \circ f_{X_n}(s),\quad\text{and}\quad f_{n+1,n}(s):=s;$$
$$\eta(n):=\eta_{X_n}.$$
For any $l\in\N_0$, $n\geq l+1$ and $s\in[0,1)$, consider the following random variables
\begin{equation*}
	q_{l,n}(s):=1-f_{X_{l+1}}\circ\cdots\circ f_{X_{n}}(s)=1-f_{l+1,n}(s).
\end{equation*}
To simplify the notation, for any $n\in\N$ and $j\in\X$, let
$$q_n:=q_{0,n}(0)\qquad\text{and}\qquad q_n(j):=q_{0,n}(a_j).$$
Under Condition \ref{Con3}, we have for any $i\in\X$ and $s\in[0,1)$, $a_i\in(0,1)$, $f_i(s)\in(0,1)$ and the function $s\mapsto f_{l+1,n}(s)$ is strictly increasing. Consequently, we can find that for any $l\in\N_0$, $n\geq l+1$ and $j\in \X$,
\begin{equation}\label{qn-range}
0<q_{l,n}(a_j)<q_{l,n}(0)<1.
\end{equation}

We have the following lemma, which gives explicit expressions for $q^{-1}_{l,n}(s)$ and $q^{-1}_n$.
\begin{lemma}\label{Lem-exp}
	For any $l\in\N_0$, $n\geq l+1$ and $s\in[0,1)$,
	\begin{equation}\label{qn-inverse}
		q^{-1}_{l,n}(s)=\frac{e^{S_l-S_n}}{1-s}+\sum_{k=l}^{n-1}\eta(k+1)\,e^{S_l-S_k},
	\end{equation}
	In particular, for any $n\in\N$,
	\begin{equation}\label{qn-exp}
		q^{-1}_n=e^{-S_n}+\sum_{k=0}^{n-1}\eta(k+1)\,e^{-S_k}.
	\end{equation}
	If assume in addition Condition \ref{Con3}, then the random variable $\eta(k)$ satisfies, for any $k\in\N$,
	\begin{equation}\label{range-eta}
		0<\eta(k)\leq\eta,
	\end{equation}
where $\displaystyle\eta:=\max_{i\in\X}\eta_i$.
\end{lemma}
\begin{pf}
	For any $i\in\X$, define a function $g_i$ on $[0,1)$ as 
	$$g_i(s):=\frac{1}{1-f_i(s)}-\frac{1}{f'_i(1)(1-s)}.$$
	Since for any $i\in\X$, 
	$$f_i(s)=1-\frac{(1-a_i)(1-s)}{1-p_is}\quad\text{and}\quad f'_i(1)=\frac{1-a_i}{1-p_i},$$
	we get for any $s\in[0,1)$,
	\begin{equation}\label{eq-g}
		g_i(s)=\eta_i.
	\end{equation}
	Then we have
	\begin{align*}
		q^{-1}_{l,n}(s)&=\left[1-f_{X_{l+1}}\circ\cdots\circ f_{X_n}(s)\right]^{-1}\\
		&=\frac{e^{S_l-S_n}}{1-s}+\sum_{k=l+1}^{n}\left[\frac{e^{S_l-S_{k-1}}}{1-f_{k,n}(s)}-\frac{e^{S_l-S_k}}{1-f_{k+1,n}(s)}\right]\\
		&=\frac{e^{S_l-S_n}}{1-s}+e^{S_l}\sum_{k=l+1}^{n}e^{-S_{k-1}}\left[\frac{1}{1-f_{X_k}\left( f_{k+1,n}(s)\right)}-\frac{1}{f'_{X_k}(1)\left(1-f_{k+1,n}(s)\right)}\right]\\
		&=\frac{e^{S_l-S_n}}{1-s}+e^{S_l}\sum_{k=l+1}^{n}e^{-S_{k-1}}\,g_{X_k}\circ f_{k+1,n}(s). 
	\end{align*}
	Combining the last equality with (\ref{eq-g}), we come to the result (\ref{qn-inverse}). 
	
	Moreover, by (\ref{eq-g}) and the fact that under Condition \ref{Con3}, $\eta_i=\frac{p_i}{1-a_i}>0$, for any $i\in\X$; we thus obtain $0<\eta(k)\leq\eta$, for any $k\in\N$.  \hfill$\square$
\end{pf}

\begin{remark}
	As $q_n$ represents the survival probability of $Z_n$ given the environment, the expression for $q_n^{-1}$ in Lemma \ref{Lem-exp} is similar to the one for branching process in i.i.d. environment by \citet{GK01}. In their case, $f_{k,n}:=f_k\circ f_{k+1}\circ\cdots\circ f_{n-1}$, for $0\leq k\leq n-1$ and $n\geq 1$; where $f_k$ is the p.g.f of an arbitrary offspring distribution from an individual in the $k$-th generation given the environment. And the following expression is established
	$$q^{-1}_n=e^{-S_n}+\sum_{k=0}^{n-1}\eta_{k,n} e^{-S_k},$$
	where $\eta_{k,n}:=g_k(f_{k+1,n}(0))$, for $0\leq k\leq n-1$ and $n\geq 1$; and is a determinist number depending on $k$ and $n$. While, in our case, due to l.f. offspring distribution and Markovian environment, $\eta(k)$ in (\ref{qn-exp}) becomes a random variable solely depending on, $X_k$, the state of the Markov chain $X$ at time $k$. 
\end{remark}

For any $n\in\N$, consider now the random variables defined by
\begin{equation}\label{def-h}
	H_n:=\frac{\sum_{k=0}^{n-1}\eta(k+1)\,e^{-S_k}}{e^{-S_k}+\sum_{k=0}^{n-1}\eta(k+1)\,e^{-S_k}}.
\end{equation}
Using Lemma \ref{Lem-exp}, we can obtain the following lemma, which gives an expression for the quenched law of $Z_n$.
\begin{lemma}\label{quench1}
	For any $i\in\X$ and $(n,z)\in\N^2$,
	$$\P_i\left(Z_n=z\,|\,X_1,\cdots,X_n\right)=e^{-S_n}\,q_n^2\,H_n^{z-1}.$$
\end{lemma}
\begin{pf}
	Suppose that $z\in\N$. Using the same method as the one for Equation (6) in \cite{K06} (P. 156), we can obtain
	$$\P_i\left(Z_n=z\,|\,X_1,\cdots,X_n\right)=\frac{e^{-S_n}H_n^{z-1}}{\left[e^{-S_n}+\sum_{k=0}^{n-1}\eta(k+1)\, e^{-S_k}\right]^2}.$$ 	
	By (\ref{qn-exp}) of Lemma \ref{Lem-exp}, we can find immediately the result. \hfill$\square$
\end{pf}

For any $n\in\N$ and $j\in\X$, let
$$G_n(j):=1-e^{-\rho(j)}\,q_n(j)\,e^{-S_n}.$$
By (\ref{qn-inverse}), we have 
\begin{equation}\label{exp-qnj}
	q_n(j)=\left[\frac{e^{-S_n}}{1-a_j}+\sum_{k=0}^{n-1}\eta(k+1)e^{-S_k}\right]^{-1}.
\end{equation}
Since 
\begin{equation}\label{eq-exp}
	e^{-\rho(j)}=m^{-1}_j,
\end{equation}
we have the following formula
\begin{equation}\label{exp-gn}
	G_n(j)=1-m_j^{-1}\,q_n(j)\,e^{-S_n}.
\end{equation}

Using Lemma \ref{quench1} and the change of measure, we can obtain a general expression of the annealed law, as stated in the following lemma.

\begin{lemma}\label{Lem-ann}
	For any $(i,j)\in\X^2$, $(n,z)\in\N^2$ and $\l\in\R$,
	$$\P_i(Z_{n+1}=z,X_{n+1}=j)=\frac{k^{n+1}(\l)\,\nu_\l(i)}{m_j^{\l+1}\,\nu_\l(j)}\,\w{\E}_{\l,i}\left[e^{-(\l+1)S_n}\,q_n^2(j)\,G_n^{z-1}(j);\,X_{n+1}=j\right].$$
\end{lemma}  
\begin{pf}
	Fix $(i,j)\in\X^2$ and $(n,z)\in\N^2$. On the one hand, by Lemma \ref{quench1}, we have  
	\begin{equation*}
		\P_i(Z_{n+1}=z,X_{n+1}=j)=\E_i\left(e^{-S_{n+1}}\,q^2_{n+1}\,H^{z-1}_{n+1}\,;\,X_{n+1}=j\right).
	\end{equation*}
Using the change of measure (\ref{eq-change}) and taking therein
\begin{equation}\label{eq} g(X_1,\ldots,X_{n+1})=\frac{e^{-(\l+1)S_{n+1}}\,q^2_{n+1}\,H^{z-1}_{n+1}\,\1_{\{X_{n+1}=j\}}}{\nu_{\l}(X_{n+1})},
\end{equation}
we obtain that for any $\l\in\R$,
\begin{equation}\label{eq-1}
\P_i(Z_{n+1}=z,X_{n+1}=j)=k^{n+1}(\l)\,\nu_\l(i)\,\w{\E}_{\l,i}\left[g(X_1,\ldots,X_{n+1})\right].
\end{equation}

On the other hand, by (\ref{qn-exp}) and (\ref{def-h}), we have
$$H_n=1-q_ne^{-S_n}.$$  
Using the last equality, (\ref{eq-exp}), (\ref{eq}) and the fact that on the event $\{X_{n+1}=j\}$, we have $q_{n+1}=q_n\left(f_{X_{n+1}}(0)\right)=q_n(j)$ and $S_{n+1}=S_n+\rho(X_{n+1})=S_n+\rho(j)$; we obtain
$$\w{\E}_{\l,i}\left[g(X_1,\ldots,X_{n+1})\right]=\frac{1}{m_j^{\l+1}\,\nu_\l(j)}\,\w{\E}_{\l,i}\left[e^{-(\l+1)S_n}\,q_n^2(j)\,G^{z-1}_n(j)\,;\,X_{n+1}=j\right].$$
Combining the last equality with (\ref{eq-1}), we thus obtain immediately the result. \hfill$\square$
\end{pf}

\section{Proofs in the strongly supercritical case} \label{sec-strong}
\setcounter{equation}{0}
Assume the hypotheses of Theorem \ref{Thm1}, that is, Conditions \ref{Con1}-\ref{Con3} and $k'(-1)>0$ (see also Figure (\ref{figure-strong}) for illustration). In this section, we will use the change of measure (\ref{eq-change}), when $\l=-1$. In other words, define the probability law $\w{\P}_{-1,i}$ through its expectation $\w{\E}_{-1,i}$ as follows
\begin{equation}\label{change-measure-1}
	\w{\E}_{-1,i}\left[g(X_1,\ldots,X_n)\right]:=\frac{\E_{i}\left[e^{- S_n}g(X_1,\ldots,X_n)\,\nu_{-1}(X_n)\right]}{k^n(-1)\nu_{-1}(i)},
\end{equation} 
for any $n\in\N$, function $g:\X^n\rightarrow \C$ and $i\in\X$. 

Applying Lemma \ref{Lem-ann} with $\l=-1$ accordingly, we have the annealed law is given by
\begin{equation}\label{exp-ann--1}
\P_i(Z_{n+1}=z,X_{n+1}=j)=\frac{k^{n+1}(-1)\,\nu_{-1}(i)}{\nu_{-1}(j)}\,\w{\E}_{-1,i}\left[q_n^2(j)\,G_n^{z-1}(j);\,X_{n+1}=j\right].
\end{equation}

For any $j\in\X$, consider the following random variables

\begin{align}
	q_\infty&:=\left[\sum_{k=0}^{+\infty}\eta(k+1)\,e^{-S_k}\right]^{-1}\label{def-lim-0}\\
	q_\infty(j)&:=\left[\frac{1}{1-a_j}+\sum_{k=0}^{+\infty}\eta(k+1)\,e^{-S_k}\right]^{-1}.\label{def-lim-q}
\end{align}

\begin{lemma}\label{Lem3}
	Assume that the conditions of Theorem \ref{Thm1} are satisfied. Then for any $(i,j)\in\X^2$ and $z\in\N$, the following convergences hold.
	\begin{enumerate}
		\item[(1)] $\displaystyle\lim_{n\rightarrow+\infty}q_n(j)=q_{\infty}\in(0,+\infty)\qquad \w{\P}_{-1,i}\mbox{-- a.s.}$
		\item[(2)] 
		$\displaystyle\lim_{n\rightarrow+\infty}G^{z-1}_n(j)=1\qquad \w{\P}_{-1,i}\mbox{-- a.s.}$
		\item[(3)] 
		$\displaystyle\lim_{n\rightarrow+\infty}\w{\E}_{-1,i}\left|q^2_n(j)\,G^{z-1}_n(j)-{q^2_\infty}\right|=0.$
	\end{enumerate}
\end{lemma}

\begin{pf}
Fix $(i,j)\in\X^2$ and $z\in\N$.
\begin{enumerate}
\item[(1)] Applying the law of large numbers to the finite Markov walk $\left(S_k\right)_{k\geq0}$, we obtain
	\begin{equation}\label{LLN}
		\lim_{n\rightarrow+\infty}\frac{S_k}{k}=\w{\bfn}_{-1}(\rho) \qquad \w{\P}_{-1,i}\mbox{-- a.s.}
	\end{equation}
By (2.36) in Lemma 2.15 of \cite{GLL19} and the strongly supercritical condition $k'(-1)>0$, we get 
\begin{equation}\label{positive-drift}
	\w{\bfn}_{-1}(\rho)=\frac{k'(-1)}{k(-1)}>0;
\end{equation} 
i.e. the Markov walk $\left(S_k\right)_{k\geq0}$ has positive drift. Let $M_n:=\sum_{k=0}^{n-1}\eta(k+1)\,e^{-S_k}$. Note that all the terms in $M_n$ are positive. Therefore, from (\ref{LLN}) and (\ref{positive-drift}), we can prove that $M_n$ is bounded, for any $n\in\N$. Applying monotone convergence theorem for the sequence $(M_n)_{n\geq1}$, we obtain $M_n$ converges $\w{\P}_{-1,i}$--a.s. to $\sum_{k=0}^{+\infty}\eta(k+1)\,e^{-S_k}$, as $n\rightarrow+\infty$. This convergence implies
\begin{equation}\label{cv-sn}
	\lim_{n\rightarrow+\infty}e^{-S_n}=0\qquad \w{\P}_{-1,i}\mbox{-- a.s.}
\end{equation} 
Consequently, by (\ref{exp-qnj}), we obtain
\begin{equation*}
	\lim_{n\rightarrow+\infty}q_n^{-1}(j)=q^{-1}_{\infty}\qquad \w{\P}_{-1,i}\mbox{-- a.s.}
\end{equation*}
By (\ref{def-lim-0}), we have $q^{-1}_{\infty}>0$. Therefore, $q_\infty\in(0,+\infty)$.

\item[(2)] From Statement (1), (\ref{cv-sn}), (\ref{exp-gn}) and (\ref{qn-range}), we have
$$\lim_{n\rightarrow +\infty}G_n(j)=1\qquad \w{\P}_{-1,i}\mbox{-- a.s.}$$
It follows that the $\w{\P}_{-1,i}$-- a.s. convergence in Statement (2) is true, by continuity of the function $x\mapsto x^{z-1}$ on $\R$. 

\item[(3)] By Condition \ref{Con3}, we have $m_j>1$, for any $j\in\X$; and so 
$$e^{-S_n}=\left(\prod_{i=1}^{n}m_{X_i}\right)^{-1}\in(0,1).$$
Combining it with (\ref{exp-gn}) and  (\ref{qn-range}), we have $G_n(j)\in(0,1)$, for any $n\in\N$; we thus obtain
\begin{equation*}
	0<q^2_n(j)\,G^{z-1}_n(j)<1.
\end{equation*}
Therefore, by Statements (1) and (2), the last inequality and Lebesgue dominated convergence theorem, we have
$$\lim_{n\rightarrow+\infty}\w{\E}_{-1,i}\left|q^2_n(j)\,G^{z-1}_n(j)-q^2_{\infty}\right|=0.$$ \hfill$\square$
\end{enumerate} 
\end{pf}

\begin{lemma}\label{Lem4}
Assume that the conditions of Theorem \ref{Thm1} are satisfied. Then for any $(i,j,k)\in\X^3$ and $z\in\N$, 
$$\lim_{n\rightarrow +\infty}\w{\E}_{-1,i}\left[q_n^2(j)\,G_n^{z-1}(j);\,X_{n+1}=k\right]=\w{\bfn}_{-1}(k)\,\w{\E}_{-1,i}\left(q^2_{\infty}\right).$$
\end{lemma}

\begin{pf}
Let $m\in\N$. For any $(i,j,k)\in\X^3$ and $n\geq m$, we have 
\begin{equation}\label{eq1}
	\w{\E}_{-1,i}\left[q_n^2(j)\,G_n^{z-1}(j);\,X_{n+1}=k\right]=I_1+I_2,
\end{equation}
where 
$$I_1:=\w{\E}_{-1,i}\left[q_m^2(j)\,G_m^{z-1}(j)\,;\,X_{n+1}=k\right]$$
and 
$$I_2:=\w{\E}_{-1,i}\left[q_n^2(j)\,G_n^{z-1}(j)-q_m^2(j)\,G_m^{z-1}(j)\,;\,X_{n+1}=k\right].$$
Let's conduct analysis on $I_1$ and $I_2$, respectively.

By Markov property, 
$$I_1=\w{\E}_{-1,i}\left[q_m^2(j)\,G_m^{z-1}(j)\,{\w{\bfP}_{-1}}^{n-m+1}\left(X_m,\,k\right)\right].$$
Using (\ref{invariant}), we have 
$$\lim_{n\rightarrow +\infty}{\w{\bfP}_{-1}}^{n-m+1}(X_m,k)=\w{\bfn}_{-1}(k)\qquad \w{\P}_{-1,i}\mbox{-- a.s.}$$
Therefore, using Lemma \ref{Lem3} (3), we have
\begin{align}
	\lim_{m\rightarrow+\infty}\lim_{n\rightarrow+\infty}I_1&=\w{\bfn}_{-1}(k)\lim_{m\rightarrow+\infty}\w{\E}_{-1,i}\left[q_m^2(j)\,G_m^{z-1}(j)\right]\notag\\
	&=\w{\bfn}_{-1}(k)\,\w{\E}_{-1,i}\left(q^2_{\infty}\right).\label{eq2}
\end{align}
Using Lemma \ref{Lem3} (3) again, we have
\begin{align*}
\lim_{m\rightarrow+\infty}\lim_{n\rightarrow+\infty}|I_2|&\leq\lim_{m\rightarrow+\infty}\lim_{n\rightarrow+\infty}\w{\E}_{-1,i}\left|q_n^2(j)\,G_n^{z-1}(j)-q_m^2(j)\,G_m^{z-1}(j)\right|\\
&=\lim_{m\rightarrow+\infty}\w{\E}_{-1,i}\left|q_{\infty}^2-q_m^2(j)\,G_m^{z-1}(j)\right|\\
&=0.
\end{align*}
Combining this with (\ref{eq2}) and (\ref{eq1}), we obtain immediately the result of the lemma.  \hfill$\square$
\end{pf}

\begin{pot1}
	Taking limit in both sides of (\ref{exp-ann--1}), as $n\rightarrow+\infty$, and applying Lemma \ref{Lem4} with $k=j$, we can obtain 
	$$\P_i(Z_n=z\,;\,X_n=j)\sim k^n(-1)\, u_0(i)\,v(j),\quad n\rightarrow+\infty;$$
	where $$u_0(i):=\nu_{-1}(i)\,\w{\E}_{-1,i}\left(q^2_{\infty}\right)\quad\text{and}\quad v(j):=\frac{\w{\bfn}_{-1}(j)}{\nu_{-1}(j)}.$$
Evidently, $v(j)>0$, for any $j\in\X$. By Lemma \ref{Lem3} (1), we get $u_0(i)>0$, for any $i\in\X$. \hfill$\square$
\end{pot1}

\begin{pot2}
	Fix $(i,j)\in\X^2$. Suppose that for a given $c\in\N$, $z\in\{1,2,\ldots,c\}$. Then we have
	\begin{equation*}
	\P_i\left(Z_n=z\,|\,1\leq Z_n\leq c\,,\,X_n=j\right)=\frac{\P_i(Z_n=z,\,X_n=j)}{\sum_{m=1}^c\P_i(Z_n=m,\,X_n=j)}.
	\end{equation*}
According to Theorem \ref{Thm1}, we have
$$\lim_{n\rightarrow+\infty}k^{-n}(-1)\,\P_i(Z_n=z,\,X_n=j)=u_0(i)\,v(j)>0.$$
Hence, we obtain
\begin{align*}
	\lim_{n\rightarrow+\infty}\P_i\left(Z_n=z\,|\,1\leq Z_n\leq c\,;\,X_n=j\right)=&\lim_{n\rightarrow+\infty}\frac{k^{-n}(-1)\,\P_i(Z_n=z,\,X_n=j)}{\sum_{m=1}^ck^{-n}(-1)\,\P_i(Z_n=m,\,X_n=j)}\\
	=&\frac{u_0(i)\,v(j)}{c\,u_0(i)\,v(j)}=\frac{1}{c}.
\end{align*} \hfill$\square$
\end{pot2}

\section{Proofs in the intermediately supercritical case} \label{sec-inter}
\setcounter{equation}{0}
Assume the hypotheses of Theorem \ref{Thm2}, that is Conditions \ref{Con1}-\ref{Con3} and $k'(-1)=0$ (see also Figure (\ref{figure-intermediate}) for illustration). In this section, we will use the same change of measure as in strongly supercritical case (see also (\ref{change-measure-1}), Section \ref{sec-strong}). Similarly, applying Lemma \ref{Lem-ann}, with $\l=-1$ and $z=1$, we have
\begin{equation}\label{exp-ann-1-1}
	\P_i(Z_{n+1}=1,X_{n+1}=j)=\frac{k^{n+1}(-1)\,\nu_{-1}(i)}{\nu_{-1}(j)}\,\w{\E}_{-1,i}\left[q_n^2(j)\,;\,X_{n+1}=j\right].
\end{equation}

To prove Theorem \ref{Thm2}, some important properties for Markov walks conditioned to stay positive will be used. For their detailed properties, one may refer to Section 2.3 of \cite{GLL19}. For any $y\in\R$, define the first time when the Markov walk $(y+S_n)_{n\geq0}$ becomes non-positive $\tau_y$ by 
$$\tau_y:=\inf\{k\in\N\,:\,y+S_k\leq0\}.$$
By Lemma 2.15 of \cite{GLL19} and the intermediately supercritical condition $k'(-1)=0$, we have
$$\w{\bfn}_{-1}(\rho)=\frac{k'(-1)}{k(-1)}=0.$$
So the Markov walk $(S_n)_{n\geq0}$ is centred under the law $\w{\P}_{-1,i}$, for any $i\in\X$. Applying Statement 1 of Proposition 2.6 in \cite{GLL19} to the Markov kernel $\w{\bfP}_{-1}$, we can obtain that it has a non-negative harmonic function $\w{V}_{-1}$ defined on $\X\times\R$, such that
$$\w{\E}_{-1,i}\left[\w{V}_{-1}(X_n,y+S_n);\tau_y>n\right]=\w{V}_{-1}(i,y).$$ 
Based on the harmonic function $\w{V}_{-1}$, let's construct a new probability law. For any $(i,y)\in\mbox{Supp}\left(\w{V}_{-1}\right)$, define a probability law $\w{\P}^+_{-1,i,y}$ and the corresponding expectation $\w{\E}^+_{-1,i,y}$ on $\sigma\left(X_1,\ldots, X_n\right)$ by
\begin{equation}\label{def-p-+}
	\w{\E}^+_{-1,i,y}\left[g(X_1,\ldots,X_n)\right]:=\frac{1}{\w{V}_{-1}(i,y)}\,\w{\E}_{-1,i}\left[g(X_1,\ldots,X_n)\w{V}_{-1}\left(X_n,y+S_n\right)\,;\,\tau_y>n\right],
\end{equation}
for any $n\in\N$ and $g:\X^n\rightarrow\C$. 

We have the following lemma.
\begin{lemma}\label{Lem5}
	Assume that the conditions of Theorem \ref{Thm2} are satisfied. Then for any $(i,y)\in\mbox{Supp}\left(\w{V}_{-1}\right)$ and $j\in\X$, we have for $l=-1,\,1,\,2$,
	\begin{equation}\label{cv-qn}
		\lim_{m\rightarrow +\infty}\w{\E}^+_{-1,i,y}\left|q^l_m(j)-q^l_\infty\right|=0.
	\end{equation}
\end{lemma}
\begin{pf}
Fix $(i,y)\in\mbox{Supp}\left(\w{V}_{-1}\right)$ and $j\in\X$. Let's prove the convergences (\ref{cv-qn}) for $l=-1,\,1,\,2$ case by case.
	\begin{description}
		\item[Case 1:] When $l=-1$. By (\ref{exp-qnj}), (\ref{def-lim-0}) and (\ref{range-eta}) in Lemma \ref{Lem-exp}, we have
		\begin{align*}
			\w{\E}^+_{-1,i,y}\left|q^{-1}_m(j)-q^{-1}_\infty\right|=&\w{\E}^+_{-1,i,y}\left|\frac{e^{-S_m}}{1-a_j}-\sum_{k=m}^{+\infty}\eta(k+1)\,e^{-S_k}\right|\\
			\leq&\frac{1}{1-a_j}\,\w{\E}^+_{-1,i,y}\left(e^{-S_m}\right)+\w{\E}^+_{-1,i,y}\left[\sum_{k=m}^{+\infty}\eta(k+1)\,e^{-S_k}\right]\\
	        \leq&\frac{1}{1-a_j}\,\w{\E}^+_{-1,i,y}\left(e^{-S_m}\right)+\eta\,\w{\E}^+_{-1,i,y}\left[\sum_{k=m}^{+\infty}e^{-S_k}\right].
		\end{align*}
	Applying Lemma 2.13 of \cite{GLL19} and using Lebesgue monotone convergence theorem,
	\begin{align*}
	\w{\E}^+_{-1,i,y}\left|q^{-1}_m(j)-q^{-1}_\infty\right|\leq&\frac{c[1+y\vee0]\,e^y}{(1-a_j)\,m^{3/2}\,\w{V}_{-1}(i,y)}+\eta\,\frac{c[1+y\vee0]\,e^y}{\w{V}_{-1}(i,y)}\,\sum_{k=m}^{+\infty}k^{-3/2}\\
	=& \frac{c[1+y\vee0)]\,e^y}{\w{V}_{-1}(i,y)}\,\left[(1-a_j)^{-1}m^{-3/2}+\eta\sum_{k=m}^{+\infty}k^{-3/2}\right]\stackrel{m\rightarrow+\infty}{\longrightarrow}0.
	\end{align*}
The convergence (\ref{cv-qn}) for $l=-1$ is thus proved.
\item[Case 2:] When $l=1$. For any $m\in\N$, since $q^{-1}_m(j)>1$ by (\ref{qn-range}), then for any $\varepsilon>0$ we have
\begin{equation}\label{eq3}
	\w{\P}^+_{-1,i,y}\left(q^{-1}_{\infty}<1-\varepsilon\right)\leq\w{\P}^+_{-1,i,y}\left(q^{-1}_{\infty}-q^{-1}_{m}<-\varepsilon\right).
\end{equation}
From Case 1, $q^{-1}_m$ converges in $\w{\P}^+_{-1,i,y}$-probability to $q^{-1}_\infty$. So for any $\varepsilon>0$, we have 
\begin{equation}\label{eq4}
	\lim_{m\rightarrow+\infty}\w{\P}^+_{-1,i,y}\left(q^{-1}_{\infty}-q^{-1}_{m}<-\varepsilon\right)=0
\end{equation}
Combining (\ref{eq3}) and (\ref{eq4}), we have for any $\varepsilon>0$,
$$\w{\P}^+_{-1,i,y}\left(q^{-1}_{\infty}<1-\varepsilon\right)=0.$$
Therefore,
\begin{equation}\label{eq5}
	q_\infty\leq1 \qquad \w{\P}^+_{-1,i,y}\mbox{-- a.s.}
\end{equation}
Using (\ref{eq5}) and (\ref{qn-range}), we have $\w{\P}^+_{-1,i,y}$-- a.s.
\begin{equation*}
		\left|q_m(j)-q_\infty\right|=q_m(j)\,q_\infty\,\left|q^{-1}_m(j)-q^{-1}_\infty\right|\leq \left|q^{-1}_m(j)-q^{-1}_\infty\right|.
\end{equation*} 
Since it has been proved in Case 1 that $$\lim_{m\rightarrow+\infty}\w{\E}^+_{-1,i,y}\left|q^{-1}_m(j)-q^{-1}_\infty\right|=0,$$
it follows with the last inequality that the convergence (\ref{cv-qn}), for $l=1$, is also true.

\item[Case 3:] When $l=2$. Similarly, by (\ref{eq5}) and (\ref{qn-range}), we have $\w{\P}^+_{-1,i,y}$-- a.s.
\begin{equation*}
	\left|q^2_m-q^2_\infty\right|=\left[q_m(j)+q_\infty\right]\, \left|q_m-q_\infty\right|\leq2\left|q_m-q_\infty\right|.
\end{equation*}
And it has been proved for $l=1$, the convergence (\ref{cv-qn}) holds; it follows with the last inequality that the convergence (\ref{cv-qn}) also holds for $l=2$. \hfill$\square$
\end{description}
\end{pf}

Let $u_1$ be a function on $\mbox{Supp}\left(\w{V}_{-1}\right)$ defined by
\begin{equation}\label{eq4*}
u_1(i,y):=\w{\E}^+_{-1,i,y}\left(q^2_\infty\right).
\end{equation}
Using (\ref{def-lim-0}), (\ref{range-eta}) in Lemma \ref{Lem-exp} and Lemma 2.13 of \cite{GLL19}, we have
$$\w{\E}^+_{-1,i,y}\left(q^{-1}_\infty\right)=\w{\E}^+_{-1,i,y}\left[\sum_{k=0}^{+\infty}\eta(k+1)\,e^{-S_k}\right]\leq\eta\,\w{\E}^+_{-1,i,y}\left(\sum_{k=0}^{+\infty}e^{-S_k}\right)<+\infty.$$
Therefore, $q_\infty>0$ $\w{\P}^+_{-1,j,y}$--a.s., which implies $u_1(i,y)>0$. Moreover, from (\ref{def-lim-0}), (\ref{exp-qnj}) and (\ref{qn-range}), we obtain
$$q^{-1}_{\infty}=\sum_{k=0}^{n-1}\eta(k+1)e^{-S_k}=q^{-1}_n(j)-\frac{e^{-S_n}}{1-a_j}\geq q^{-1}_n(j)>1,$$
for any $n\in\N$ and $i\in\X$. It follows $q^2_{\infty}<1$. As a result, taking into account (\ref{eq4*}), we get $u_1(i,y)<1$. In conclusion, for any $(i,y)\in\mbox{Supp}\left(\w{V}_{-1}\right)$, 
\begin{equation}\label{range-u1}
	u_1(i,y)\in(0,1).
\end{equation}

Using the last lemma, we can obtain the following Lemma.
\begin{lemma}\label{Lem6}
 Assume that the conditions of Theorem \ref{Thm2} are satisfied. Then the following results hold.
 \begin{enumerate}
 	\item [(1)] For any $(i,y)\in\mbox{Supp}\left(\w{V}_{-1}\right)$ and $(j,k)\in\X^2$,
 	$$\lim_{m\rightarrow+\infty}\lim_{n\rightarrow+\infty}\w{\E}_{-1,i}\left[q^2_m(j)\,;\,X_{n+1}=k\,|\,\tau_y>n+1\right]=u_1(i,y)\,\w{\bfn}_{-1}(k).$$
 	\item [(2)] For any $(i,y)\in\mbox{Supp}\left(\w{V}_{-1}\right)$, $j\in\X$ and $\theta\in(0,1)$,
	$$\lim_{m\rightarrow+\infty}\limsup_{n\rightarrow+\infty}\w{\E}_{-1,i}\left[\left|q^2_m(j)-q^2_{[\theta n]}(j)\right|\,|\,\tau_y>n+1\right]=0.$$
 \end{enumerate}
\end{lemma}
\begin{pf}
	\begin{enumerate}
		\item [(1)] Fix $(i,y)\in\mbox{Supp}\left(\w{V}_{-1}\right)$ and $(j,k)\in\X^2$. Applying Lemma 2.12 of \cite{GLL19} with $g(X_1,\ldots,X_m)=q^2_m(j)$, we obtain
		$$\lim_{n\rightarrow +\infty}\w{\E}_{-1,i}\left[q^2_m(j);\,X_{n+1}=k\,|\,\tau_y>n+1\right]=\w{\E}^+_{-1,i,y}\left[q^2_m(j)\right]\w{\bfn}_{-1}(k).$$
		Applying Lemma \ref{Lem5} for $l=2$, we thus obtain immediately the desired result.
		
		 \item [(2)] Fix $(i,y)\in\mbox{Supp}\left(\w{V}_{-1}\right)$, $j\in\X$ and $\theta\in(0,1)$. Let $m\in\N$ and $n\in\N$ such that $\theta n\geq m+1$. Set $\theta_n=[\theta n]$. Denote
		 $$I_3:=\w{\E}_{-1,i}\left[\left|q^2_m(j)-q^2_{\theta n}(j)\right|\,|\, \tau_y>n+1\right]$$
		and 
		$$J_n(i,y):=\w{\P}_{-1,i}(\tau_y>n).$$
		By Statement 1 of Proposition 2.7 in \cite{GLL19}, $J_n(i,y)>0$ for $n$ large enough. Using Markov property, we have
		\begin{align*}
			I_3&=J^{-1}_{n+1}(i,y)\,\w{\E}_{-1,i}\left(\left|q^2_m(j)-q^2_{\theta n}(j)\right|\,;\,\tau_y>n+1\right)\\
			&=J^{-1}_{n+1}(i,y)\,\w{\E}_{-1,i}\left(\left|q^2_m(j)-q^2_{\theta n}(j)\right|\,J_{n+1-\theta_n}(X_{\theta_n},y+S_{\theta_n})\,;\,\tau_y>\theta_n\right).
		\end{align*}
		From Statement 2 of Proposition 2.7 in \cite{GLL19}, 
		$$J_{n+1-\theta_n}(X_{\theta_n},y+S_{\theta_n})\leq \frac{c\,\left[1+(y+S_{\theta_n})\vee0\right]}{\sqrt{n+1-\theta_n}}.$$
		Therefore, by Statement 3 of Proposition 2.6 in \cite{GLL19}, (\ref{qn-range}) and (\ref{def-p-+}), we obtain
		\begin{align*}
			I_3&\leq c\,J^{-1}_{n+1}(i,y)\left[n(1-\theta)\right]^{-1/2}\,\w{\E}_{-1,i}\left[\left|q^2_m(j)-q^2_{\theta_n}(j)\right|\,(1+y+S_{\theta_n})\,;\, \tau_y>\theta_n\right]\\
			&\leq c\,J^{-1}_{n+1}(i,y)\left[n(1-\theta)\right]^{-1/2}\,\w{\E}_{-1,i}\left[\left|q^2_m(j)-q^2_{\theta_n}(j)\right|\,\left(1+\w{V}_{-1}\left(X_{\theta_n},y+S_{\theta_n}\right)\right)\,;\, \tau_y>\theta_n\right]\\
			&\leq c\,J^{-1}_{n+1}(i,y)\left[n(1-\theta)\right]^{-1/2}\,\left[\w{\P}_{-1,i}(\tau_y>\theta_n)+\w{V}_{-1}(i,y)\,\w{\E}^+_{-1,i,y}\left|q^2_m(j)-q^2_{\theta_n}(j)\right|\right].
	    \end{align*}
    Using Proposition 2.7 of \cite{GLL19} and Lemma \ref{Lem5} for $l=2$, we obtain
    \begin{align*}
    	\limsup_{n\rightarrow+\infty} I_3\leq&\limsup_{n\rightarrow+\infty}\,\frac{c\,\sqrt{n+1}}{\sqrt{n(1-\theta)}}\,\w{\E}^+_{-1,i,y}\left|q^2_m(j)-q^2_{\theta_n}(j)\right|\\
    	=&\frac{c}{\sqrt{1-\theta}}\,\w{\E}^+_{-1,i,y}\left|q^2_m(j)-q^2_{\infty}(j)\right|.
    \end{align*}
Taking $m\rightarrow+\infty$ and using Lemma \ref{Lem5} for $l=2$ again, we thus derive the result.  \hfill$\square$
    \end{enumerate}
\end{pf}

\begin{lemma}\label{Lem7}
	Assume that the conditions of Theorem \ref{Thm2} are satisfied. Then for any $(i,y)\in\mbox{Supp}\left(\w{V}_{-1}\right)$, $(j,k)\in\X^2$ and $\theta\in(0,1)$,
	$$\lim_{n\rightarrow+\infty}\w{\E}_{-1,i}\left[q^2_{[\theta n]}(j);\,X_{n+1}=k\,|\,\tau_y>n+1\right]=u_1(i,y)\,\w{\bfn}_{-1}(k).$$
\end{lemma}

\begin{pf} 
	Fix $(i,y)\in\mbox{Supp}\left(\w{V}_{-1}\right)$, $(j,k)\in\X^2$ and $\theta\in(0,1)$. Let
	\begin{equation*}
		I_4:=\w{\E}_{-1,i}\left[q^2_{[\theta n]}(j);\,X_{n+1}=k\,|\,\tau_y>n+1\right].
	\end{equation*}
Suppose that $m\in\N$ and $n\geq m,$ such that $[\theta n]\geq m$. Then we have
\begin{equation*}
	I_4=\w{\E}_{-1,i}\left[q^2_{m}(j);\,X_{n+1}=k\,|\,\tau_y>n+1\right]+I_5,
\end{equation*}
where 
$$I_5:=\w{\E}_{-1,i}\left[q^2_{[\theta n]}(j)-q^2_{m}(j);\,X_{n+1}=k\,|\,\tau_y>n+1\right].$$
Using Lemma \ref{Lem6} (2), we have $\displaystyle\limsup_{m\rightarrow+\infty}\limsup_{n\rightarrow+\infty}|I_5|=0$. Now, applying Lemma \ref{Lem6} (1), we get 
$$\lim_{n\rightarrow+\infty}I_4=u_1(i,y)\,\w{\bfn}_{-1}(k).$$ \hfill$\square$
\end{pf}

\begin{lemma}\label{Lem8}
		Assume that the conditions of Theorem \ref{Thm2} are satisfied. Then for any $(i,y)\in\mbox{Supp}\left(\w{V}_{-1}\right)$ and $(j,k)\in\X^2$,
		$$\lim_{p\rightarrow+\infty}\w{\E}_{-1,i}\left[q^2_p(j);\,X_{p+1}=k\,|\,\tau_y>p+1\right]=u_1(i,j)\,\w{\bfn}_{-1}(k).$$
\end{lemma}
\begin{pf}
	Fix $(i,y)\in\mbox{Supp}\left(\w{V}_{-1}\right)$ and $(j,k)\in\X^2$. For any $n\in\N$ and $\theta\in(0,1)$, set $n=[p/\theta]+1$. Note that $p=[\theta n]$. We have for any $p\in\N$,
\begin{equation}\label{eq6}
	\w{\E}_{-1,i}\left[q^2_p(j);\,X_{p+1}=k\,|\,\tau_y>p+1\right]=I_6+I_7,
\end{equation}
	where 
	$$I_6:=\frac{\w{\E}_{-1,i}\left[q^2_p(j);\,X_{p+1}=k,\,\tau_y>n+1\right]}{\w{\P}_{-1,i}\left(\tau_y>p+1\right)}$$
and
   $$I_7:=\frac{\w{\E}_{-1,i}\left[q^2_p(j);\,X_{p+1}=k,\, p+1<\tau_y\leq n+1\right]}{\w{\P}_{-1,i}\left(\tau_y>p+1\right)}.$$
  
On the one hand, by Lemme \ref{Lem7} and Statement 1 of Proposition 2.7 in \cite{GLL19}, we have
\begin{equation}\label{eq7}
	I_6=\frac{\w{\P}_{-1,i}\left(\tau_y>n+1\right)}{\w{\P}_{-1,i}\left(\tau_y>p+1\right)}\,\w{\E}_{-1,i}\left[q^2_p(j);\,X_{p+1}=k\,|\,\tau_y>n+1\right]\stackrel{p\rightarrow+\infty}{\longrightarrow} u_1(i,y)\,\w{\bfn}_{-1}(k)\,\sqrt{\theta}.
\end{equation}
On the other hand, by (\ref{qn-range}) and Statement 1 of Proposition 2.7 in \cite{GLL19}, we have
\begin{equation}\label{eq8}
	I_7\leq\frac{\w{\P}_{-1,i}\left(p+1<\tau_y\leq n+1\right)}{\w{\P}_{-1,i}\left(\tau_y>p+1\right)}=1-\frac{\w{\P}_{-1,i}\left(\tau_y>n+1\right)}{\w{\P}_{-1,i}\left(\tau_y>p+1\right)}\stackrel{p\rightarrow+\infty}{\longrightarrow}1-\sqrt{\theta}.
\end{equation}

Combining (\ref{eq6}), (\ref{eq7}) and (\ref{eq8}), we have for any $\theta\in(0,1)$, 
$$\left|\lim_{p\rightarrow+\infty}\w{\E}_{-1,i}\left[q^2_p(j);\,X_{p+1}=k\,|\,\tau_y>p+1\right]-u_1(i,y)\,\w{\bfn}_{-1}(k)\,\sqrt{\theta}\right|\leq 1-\sqrt{\theta}.$$
Taking the limit as $\theta\rightarrow1$, we thus obtain the result. \hfill$\square$
\end{pf}

\begin{lemma}\label{Lem9}
Assume that the conditions of Theorem \ref{Thm2} are satisfied. Then there exists a positive function $u_2$ on $\X$ such that for any $(i,j,k)\in\X^3$, as $n\rightarrow+\infty$,
$$\w{\E}_{-1,i}\left[q^2_n(j);\,X_{n+1}=k\right]\sim\frac{u_2(i)\,\w{\bfn}_{-1}(k)}{\sqrt{n}}.$$
\end{lemma}
\begin{pf}
	Fix $(i,j,k)\in\X^3$. For any $y\in\R$ and $n\in\N$, we have
\begin{align}
	0&\leq \w{\E}_{-1,i}\left[q^2_n(j);\,X_{n+1}=k\right]-\w{\E}_{-1,i}\left[q^2_n(j);\,X_{n+1}=k,\,\tau_y>n+1\right]\notag\\
	&\leq\w{\E}_{-1,i}\left[q^2_n(j);\,\tau_y\leq n+1\right]\label{eq9}
\end{align}
On the one hand, by (\ref{qn-range}), we have
\begin{equation}\label{eq10}
	\w{\E}_{-1,i}\left[q^2_n(j);\,\tau_y\leq n+1\right]\leq \w{\E}_{-1,i}\left[q^2_n(j);\,\tau_y\leq n\right]+\w{\P}_{-1,i}\left(\tau_y=n+1\right).
\end{equation}
By (\ref{qn-range}), we have for any $n\in\N$, $q_n(j)<q_n$. Moreover, since $(q_n)_{n\geq1}$ is non-increasing and from (\ref{qn-exp}), we obtain $$q_n(j)\leq\min_{1\leq m\leq n}q_m\leq e^{\min_{1\leq m\leq n} S_m}$$
and so
$$q^2_n(j)\leq e^{2\,\min_{1\leq m\leq n} S_m}.$$
Therefore, 
\begin{align*}
	\w{\E}_{-1,i}\left[q^2_n(j);\,\tau_y\leq n\right]\leq&\, e^{-2y}\,\w{\E}_{-1,i}\left[e^{2\min_{1\leq m\leq y}\,y+S_m};\,\tau_y\leq n\right]\\
	\leq&\, e^{-2y}\,\sum_{p=0}^{+\infty}e^{-p}\,\w{\P}_{-1,i}\left[-\frac{p+1}{2}<\min_{1\leq m\leq n}\,y+S_m\leq-\frac{p}{2}\,,\,\tau_y\leq n\right]\\
	\leq&\,e^{-2y}\,\sum_{p=0}^{+\infty}e^{-p}\,\w{\P}_{-1,i}\left(\tau_{y+\frac{p+1}{2}}> n\right).
\end{align*}
Using Statement 2 of Proposition 2.7 in \cite{GLL19}, we have
\begin{equation}\label{eq11}
	\w{\E}_{-1,i}\left[q^2_n(j);\,\tau_y\leq n\right]\leq \frac{c\,e^{-2y}\left[1+y\vee0\right]}{\sqrt{n}}.
\end{equation}
By Statement 3 of Proposition 2.6 in \cite{GLL19}, there exists $y_0\in\R$ such that for any $y\geq y_0$, $\w{V}_{-1}(i,y)>0$, i.e. $(i,y)\in\mbox{Supp}\left(\w{V}_{-1}\right)$. Now using Statement 1 of Proposition 2.7 in \cite{GLL19}, we obtain for any $y\geq y_0$, 
\begin{align}
	\lim_{n\rightarrow+\infty}\sqrt{n}\,\w{\P}_{-1,i}\left(\tau_y=n+1\right)=&\lim_{n\rightarrow+\infty}\sqrt{n}\,\w{\P}_{-1,i}\left(\tau_y>n\right)-\lim_{n\rightarrow+\infty}\sqrt{n}\,\w{\P}_{-1,i}\left(\tau_y>n+1\right)\notag\\
	=&\,0. \label{eq12}
\end{align}
Combining (\ref{eq10}), (\ref{eq11}) and (\ref{eq12}), we get for any $y\geq y_0$, 
\begin{equation}\label{eq13}
\sqrt{n}\,\w{\E}_{-1,i}\left[q^2_n(j);\,\tau_y\leq n+1\right]\leq\,c\,e^{-2y}\left[1+y\vee0\right].
\end{equation}

On the other hand, using Lemma \ref{Lem8} and Statement 1 of Proposition 2.7 in \cite{GLL19}, 
\begin{equation}\label{eq14}
	\lim_{n\rightarrow+\infty}\sqrt{n}\,\w{\E}_{-1,i}\left[q^2_n(j);\,X_{n+1}=k,\,\tau_y>n+1\right]=u_1(i,y)\,\w{\bfn}_{-1}(k)\,\frac{\sqrt{2}\,\w{V}_{-1}(i,y)}{\sqrt{\pi}\,\w{\sigma}_{-1}},
\end{equation}
where $\displaystyle\w{\sigma}_{-1}^2:=\w{\bfn}_{-1}\left(\rho^2\right)-\w{\bfn}_{-1}^2\left(\rho\right)+2\sum_{n=1}^{+\infty}\left[\w{\bfn}_{-1}\left(\rho\w{\bfP}^n_{-1}\rho\right)-\w{\bfn}^2_{-1}\left(\rho\right)\right]>0$ (See also (2.37) in \cite{GLL19}). Denote 
$$I(i,j,k)=\liminf_{n\rightarrow+\infty}\sqrt{n}\,\w{\E}_{-1,i}\left[q^2_n(j);\,X_{n+1}=k\right]$$
and 
$$J(i,j,k)=\limsup_{n\rightarrow+\infty}\sqrt{n}\,\w{\E}_{-1,i}\left[q^2_n(j);\,X_{n+1}=k\right].$$
Using (\ref{eq9}), (\ref{eq13}) and (\ref{eq14}), we have for any $y\geq y_0$, 
\begin{align}\label{eq15}
	\frac{\sqrt{2}\,\w{V}_{-1}(i,y)}{\sqrt{\pi}\,\w{\sigma}_{-1}} u_1(i,y)\,\w{\bfn}_{-1}(k)\leq I(i,j,k)\leq& J(i,j,k)\notag\\
	\leq&\frac{\sqrt{2}\,\w{V}_{-1}(i,y)}{\sqrt{\pi}\,\w{\sigma}_{-1}} u_1(i,y)\,\w{\bfn}_{-1}(k)+c\,e^{-2y}\left[1+y\vee0\right]. 
\end{align}
By (\ref{eq14}), observe that $\displaystyle y\mapsto \frac{\sqrt{2}\,\w{V}_{-1}(i,y)}{\sqrt{\pi}\,\w{\sigma}_{-1}}\,u_1(i,y)$ is non-decreasing and by (\ref{eq15}), the function is bounded by $\displaystyle \frac{I(i,j,k)}{\w{\bfn}_{-1}(k)}$. Consequently, the limit
$$u_2(i):=\lim_{y\rightarrow+\infty}\frac{\sqrt{2}\,\w{V}_{-1}(i,y)}{\sqrt{\pi}\,\w{\sigma}_{-1}}\,u_1(i,y)$$
exists. Moreover, by (\ref{range-u1}), we obtain for any $y\geq y_0$,
\begin{equation*}
	u_2(i)\geq \frac{\sqrt{2}\,\w{V}_{-1}(i,y)}{\sqrt{\pi}\,\w{\sigma}_{-1}}\,u_1(i,y)>0.
\end{equation*}
Taking the limit as $y\rightarrow+\infty$ in (\ref{eq15}), we obtain
$$I(i,j,k)=J(i,j,k)=u_2(i)\,\w{\bfn}_{-1}(k).$$ 
\hfill$\square$
\end{pf}

\begin{pot3}
	Define the functions $u(i)$ and $v(j)$, both on $X$, by
	\begin{equation*}
		u(i):=u_2(i)\,\nu_{-1}(i)\quad\text{and}\quad v(j):=\frac{\w{\bfn}_{-1}(j)}{\nu_{-1}(j)}.
	\end{equation*}
	Evidently, $v(j)>0$, for any $j\in\X$. Applying Lemma \ref{Lem9} with $k=j$ and using (\ref{exp-ann-1-1}), we can immediately find the result. \hfill$\square$
\end{pot3}

\section{Proofs in the weakly supercritical case} \label{sec-weak}
\setcounter{equation}{0}
Assume the hypotheses of Theorem \ref{Thm1}, that is, Conditions \ref{Con1}-\ref{Con3} and $k'(-1)<0$ (see also Figure (\ref{figure-weak}) for illustration). By Lemma 2.15 of \cite{GLL19}, the function $K:\l\mapsto \ln\left(k(\l)\right)$ is strictly convex on $\R$ and satisfies for any $\l\in\R$,
\begin{equation}\label{eq-K}
	K'(\l)=\frac{k'(\l)}{k(\l)}=\w{\bfn}_\l(\rho).
\end{equation}  
As a result, the function $\l\mapsto K'(\l)$ is increasing on $\R$. Moreover, by weakly supercritical condition $k'(-1)<0$, Condition \ref{Con3} and (\ref{eq-K}), we have $K'(-1)<0$ and $K'(0)>0$. Therefore, we obtain that there exists a unique real value $\l\in(-1,0)$ such that
\begin{equation}\label{eq-K'}
	K'(\l)=\w{\bfn}_\l(\rho)=0.
\end{equation}
In the sequel, suppose that $\l$ is the unique real value satisfying (\ref{eq-K'}). For any $i\in\X$, let $\w{\P}_{\l,i}$ and $\w{\E}_{\l,i}$ be the probability and expectation respectively generated by the finite dimensional distributions of the Markov chain $(X_n)_{n\geq0}$ with Markov kernel $\w{\bfP}_\l$ and starting at $X_0=i$. That is,
\begin{equation}
	\w{\E}_{\l,i}\left[g(X_1,\ldots,X_n)\right]:=\frac{\E_{i}\left[e^{\l S_n}g(X_1,\ldots,X_n)\,\nu_\l(X_n)\right]}{k^n(\l)\,\nu_\l(i)},
\end{equation}
for any $n\in\N$, function $g:\X^n\rightarrow \C$ and $i\in\X$.
So by (\ref{eq-K'}), the Markov walk $(S_n)_{n\geq0}$ is centred under the law $\w{\P}_{\l,i}$, for any $i\in\X$. Applying Statement 1 of Proposition 2.6 in \cite{GLL19} to the Markov kernel $\w{\bfP}_{\l}$, we can obtain that it has a non-negative harmonic function $\w{V}_{\l}$ defined on $\X\times\R$, such that
\begin{equation}\label{harmonic}
	\w{\E}_{\l,i}\left[\w{V}_{\l}(X_n,y+S_n);\tau_y>n\right]=\w{V}_{\l}(i,y).
\end{equation} 

Let us introduce the dual Markov chain, which is independent with the original chain $X$. Consider the dual Markov chain $X^*=(X_n^*)_{n\geq0}$, with the transition probabilities ${\w{\bfP}_\l}^*$ defined by
\begin{equation}\label{eq-transition}
	\w{\bfP}_\l^*(i,j):=\frac{\w{\bfn}_\l(j)}{\w{\bfn}_\l(i)}\w{\bfP}_\l(j,i)=\frac{\bfn_\l(j)}{\bfn_\l(i)}\frac{e^{\rho(i)}}{k(\l)}\bfP_\l(j,i).
\end{equation} 
Similarly, for any $j\in\X$, let $\w{\P}^*_{\l,j}$ and $\w{\E}^*_{\l,j}$ be the probability and expectation respectively generated by the finite dimensional distributions of the dual Markov chain $X^*$ with Markov kernel $\w{\bfP}^*_\l$ and starting at $X^*_0=j$. Define the associated dual Markov walk $(S_n^*)_{n\geq0}$ by 
$$S^*_0=0\quad\text{and}\quad S^*_n=-\sum_{k=1}^n\rho(X^*_k),\quad \forall n\in\N.$$
From Remark \ref{Rem1}, $\w{\bfP}_\l$ satisfies Conditions \ref{Con1} and \ref{Con2}. Therefore, according to Lemma 2.4 in \cite{GLL19}, $\w{\bfP}_\l^*$ also satisfies Conditions \ref{Con1} and \ref{Con2}. From (\ref{eq-K'}), we have
$$\w{\bfn}_\l(-\rho)=0.$$
So the Markov walk $(S^*_n)_{n\geq0}$ is also centred under $\w{\P}^*_{\l,j}$, for any $j\in\X$. For any $z\in\R$, let $\tau^*_z$ the first time when the Markov walk $(z+S^*_n)_{n\geq0}$ becomes non-positive, i.e.
$$\tau^*_z:=\inf\{k\in\N\,:\,z+S^*_k\leq0\}.$$ 
Applying Statement 1 of Proposition 2.6 in \cite{GLL19} to the Markov kernel $\w{\bfP}_{\l}^*$, it has a non-negative harmonic function $\w{V}^*_{\l}$ defined on $\X\times\R$, such that
\begin{equation}\label{harmonic*}
	\w{\E}^*_{\l,j}\left[\w{V}^*_{\l}(X^*_n,z+S^*_n);\tau^*_z>n\right]=\w{V}^*_{\l}(j,z).
\end{equation}

For any $n\in\N$ and $s\in [0,1)$, set
\begin{equation}\label{q*n}
	q^*_n(s):=e^{S^*_n}\left[1-f_{X^*_n}\circ\cdots\circ f_{X^*_1}(s)\right].
\end{equation}
To simplify the notation, for any $j\in\X$, let 
$$q^*_n(j):=q^*_n(a_j).$$
For any $n\in\N_0$, $i\in\X$ and $s\in[0,1)$, set 
$$f^*_{n}(s):=f^*_{X_n}\circ\cdots \circ f^*_{X_1}(s)\quad\text{and}\quad f^*_{0}(s):=s.$$
For any $k\in\N$, let
$$\eta^*(k):=\eta_{X^*_k}.$$

For any $(i,j)\in\X^2$, denote by $\w{\P}_{\l,i,j}$ the probability law  generated by the finite dimensional distribution of the two dimensional Markov chain $(X_n,X^*_n)_{n\geq0}$ starting at $(X_0,X^*_0)=(i,j)$. Let $\w{\E}_{\l,i,j}$ be the corresponding expectation. Recall that the dual chain $X^*$ is independent with the chain $X$. It follows
$$\w{\P}_{\l,i,j}=\w{\P}_{\l,i}\times \w{\P}^*_{\l,j}.$$ 

We have the following formulas for ${q^*_n}^{-1}(s)$ and ${q^*_n}^{-1}(j)$, respectively.
\begin{lemma}\label{Lem-exp*}
	For $s\in[0,1)$ and $n\in\N$,
	\begin{equation}\label{q*n-inverse}
		{q^*_n}^{-1}(s)=\frac{1}{1-s}+\sum_{k=1}^n\eta^*(k)\,e^{-S^*_k},
	\end{equation}
	In particular, for any $n\in\N$,
	\begin{equation}\label{q*n-exp}
		{q^*_n}^{-1}(j)=\frac{1}{1-a_j}+\sum_{k=1}^n\eta^*(k)\,e^{-S^*_k}.
	\end{equation}
	If assume in addition Condition \ref{Con3}, then the random variable $\eta^*(k)$ satisfies, for any $k\in\N$,
	\begin{equation}\label{range-eta*}
		0<\eta^*(k)\leq\eta,
	\end{equation}
	where $\displaystyle\eta:=\max_{i\in\X}\eta_i$.
\end{lemma}
The poof of the lemma above is similar to Lemma \ref{Lem-exp}, we thus leave it to the readers.

We have the following lemmas.
\begin{lemma}\label{lem2-1}
	Assume that the conditions of Theorem \ref{Thm1} are satisfied. Then for any $(i,y)\in(\X,\R)$, $k\in\N$ and $n\geq k+1$,
	$$n^{3/2}\,\w{\E}_{\l,i}\left[e^{-S_k}\,e^{-(\l+1)S_n};\tau_y>n\right]\leq c\,e^{(2+\l)y}\left[1+y\vee 0\right]\frac{n^{3/2}}{(n-k)^{3/2}\,k^{3/2}}.$$
\end{lemma}
\begin{pf}
	For any $m\in\N$, $(j,z)\in\X\times\R$ and $p\in\N_0$, let 
	$$J_m(j,z):=\w{\P}_{\l,j}\left(z+S_m\in(p,p+1],\,\tau_z>n\right).$$ 
	Fix $(i,y)\in(\X,\,\R)$, $k\in\N$ and $n\geq k+1$. Denote
	\begin{equation*}
		I_1:=n^{3/2}\,\w{\E}_{\l,i}\left[e^{-S_k}\,e^{-(\l+1)S_n};\tau_y>n\right].
	\end{equation*}
	Since
	\begin{align*}
		\w{\E}_{\l,i}\left[e^{-S_k}\,e^{-(\l+1)S_n};\tau_y>n\right]&=e^{(\l+1)y}\sum_{p=0}^{+\infty}\w{\E}_{\l,i}\left[e^{-S_k}\,e^{-(\l+1)(y+S_n)};\,y+S_n\in(p,\,p+1],\,\tau_y>n\right]\\
		&\leq e^{(\l+1)y}\sum_{p=0}^{+\infty}e^{-(\l+1)p}\,\w{\E}_{\l,i}\left[e^{-S_k};\,y+S_n\in(p,\,p+1],\,\tau_y>n\right].
	\end{align*}
	and by Markov property, we have
	\begin{align*}
		I_1\leq n^{3/2}\,e^{(\l+1)y}\,\sum_{p=0}^{+\infty}e^{-(\l+1)p}\,\w{\E}_{\l,i}\left[e^{-S_k}J_{n-k}(X_k,y+S_k);\,\tau_y>k\right].
	\end{align*}
	According to Statement 2 of Proposition 2.9 in \cite{GLL19}, we have for any $(j,z)\in\X\times\R$ and $p\in\N_0$,
	$$J_{n-k}(j,z)\leq\frac{c}{(n-k)^{3/2}}\,(1+p)\left[1+z\vee 0\right].$$
	Consequently,
	\begin{align*}
		I_1&\leq c\, e^{(\l+1)y}\frac{n^{3/2}}{(n-k)^{3/2}}\,\w{\E}_{\l,i}\left[e^{-S_k}(1+y+S_k);\,\tau_y>k\right]\times\sum_{p=0}^{+\infty}(1+p)\,e^{-(\l+1)p}\\
		&\leq c\, e^{(\l+1)y}\frac{n^{3/2}}{(n-k)^{3/2}}\,\w{\E}_{\l,i}\left[e^{-S_k}(1+y+S_k);\,\tau_y>k\right]\\
		&=c\, e^{(\l+2)y}\frac{n^{3/2}}{(n-k)^{3/2}}\,\sum_{p=0}^{+\infty}\w{\E}_{\l,i}\left[e^{-(y+S_k)}(1+y+S_k);\,y+S_k\in(p,\,p+1],\,\tau_y>k\right]\\
		&\leq c\, e^{(\l+2)y}\frac{n^{3/2}}{(n-k)^{3/2}}\,\sum_{p=0}^{+\infty}(2+p)\,e^{-p}\,\w{\P}_{\l,i}\left(y+S_k\in(p,\,p+1],\,\tau_y>k\right).
	\end{align*}
	Applying Statement 2 of Proposition 2.9 in \cite{GLL19} again, we obtain
	$$I_1\leq c\, e^{(\l+2)y}\,\frac{n^{3/2}(1+y\vee 0)}{(n-k)^{3/2}\,k^{3/2}}\,\sum_{p=0}^{+\infty}(2+p)\,e^{-p}.$$
	The desired result is thus obtained because of the fact that $\sum_{p=0}^{+\infty}(2+p)\,e^{-p}<+\infty$. \hfill$\square$
\end{pf}

\begin{lemma}\label{lem2-2}
	Assume that the conditions of Theorem \ref{Thm1} are satisfied. Then for any $(i,j)\in\X^2$ and $y\in\R$, 
	$$\lim_{l,m\rightarrow+\infty}\limsup_{n\rightarrow+\infty}n^{3/2}\,\w{\E}_{\l,i}\left[\left|e^{-S_{n-m}}q^{-1}_{n-m,n}(a_j)-e^{-S_{l}}q^{-1}_{l,n}(a_j)\right|e^{-(\l+1)S_n};\,\tau_y>n\right]=0.$$
\end{lemma}
\begin{pf}
	Fix $(i,j)\in\X^2$ and $y\in\R$. Denote
	\begin{equation}\label{def-I2}
		I_2:=n^{3/2}\,\w{\E}_{\l,i}\left[\left|e^{-S_{n-m}}q^{-1}_{n-m,n}(a_j)-e^{-S_{l}}q^{-1}_{l,n}(a_j)\right|e^{-(\l+1)S_n};\,\tau_y>n\right].
	\end{equation}
	Using (\ref{qn-inverse}) in Lemma \ref{Lem-exp}, we have for any $l\in\N$, $m\in\N$ and $n\geq l+m+1$, 
	\begin{equation}\label{eq-weakly}
		e^{-S_{l}}q^{-1}_{l,n}(a_j)-e^{-S_{n-m}}q^{-1}_{n-m,n}(a_j)=\sum_{k=l}^{n-m-1}\eta(k+1)e^{-S_k}>0.
	\end{equation}
	Combining this with (\ref{range-eta}) in Lemma \ref{Lem-exp} and using Lemma \ref{lem2-1}, we have
	\begin{align}
		I_2&= n^{3/2}\,\w{\E}_{\l,i}\left[e^{-(\l+1)S_n}\sum_{k=l}^{n-m-1}\eta(k+1)e^{-S_k};\,\tau_y>n\right]\notag\\
		&\leq c\eta\,e^{(2+\l)y}\left[1+y\vee 0\right]\sum_{k=l}^{n-m-1}\frac{n^{3/2}}{(n-k)^{3/2}\,k^{3/2}}\label{eq1}.
	\end{align}
	Let $n_1:=\left[\frac{n}{2}\right]$. Note that
	\begin{align*}
		\sum_{k=l}^{n-m-1}\frac{n^{3/2}}{(n-k)^{3/2}\,k^{3/2}}&=n^{3/2}\left[\sum_{k=l}^{n_1}(n-k)^{-3/2}\,k^{-3/2}+\sum_{k=n_1+1}^{n-m-1}(n-k)^{-3/2}\,k^{-3/2}\right]\\
		&\leq \frac{n^{3/2}}{(n-n_1)^{3/2}}\sum_{k=l}^{n_1}k^{-3/2}+\frac{n^{3/2}}{n_1^{3/2}}\,\sum_{k=n_1+1}^{n-m-1}(n-k)^{-3/2}\\&\leq c\sum_{k=l}^{+\infty}k^{-3/2}+c\sum_{k=m}^{+\infty}k^{-3/2}\\
		&\leq c\sum_{k=l\wedge m}^{+\infty}k^{-3/2}.
	\end{align*}
	Combining the last inequality with (\ref{eq1}), we have
	$$\limsup_{n\rightarrow+\infty}I_2 \leq c\,\eta\,e^{(2+\l)y}\left[1+y\vee 0\right]\sum_{k=l\wedge m}^{+\infty}k^{-3/2}.$$
	Taking the limits as $l\rightarrow+\infty$ and $m\rightarrow+\infty$, we come to the statement of Lemma \ref{lem2-2}. \hfill$\square$
\end{pf}

For any $(l,n)\in\N^2$ and $n\geq l+m+1$, consider the random variables
$$r_n^{(l,m)}(j):=1-f_{1,l}\left(\left[1-f'_{l+1,n-m}(1)\left(1-f_{n-m+1,n}(a_j)\right)\right]\vee0\right).$$
By Condition \ref{Con3}, we have 
\begin{equation}\label{range-rnlm}
	r_n^{(l,m)}(j)\in[0,1).
\end{equation}

\begin{lemma}\label{lem2-3}
	Assume that the conditions of Theorem \ref{Thm1} are satisfied. Then for any $(i,j)\in\X^2$ and $y\in\R$, 
	$$\lim_{l,m\rightarrow+\infty}\limsup_{n\rightarrow+\infty}n^{3/2}\,\w{\E}_{\l,i}\left[\left|q^2_n(a_j)-\left(r_n^{(l,m)}(j)\right)^2\right|e^{-(\l+1)S_n};\,\tau_y>n\right]=0.$$
\end{lemma}

\begin{pf}
	Fix $(i,j)\in\X^2$ and $y\in\R$. Since for any $j\in\X$, the function $f_j$ is increasing and convex, the function $f_{l+1,n-m}$ is also convex. So for any $(l,m)\in\N^2$ and $n\geq l+m+1$, we get
	$$f_{l+1,n}(a_j)=f_{l+1,n-m}\circ f_{n-m+1,n}(a_j)\geq \left[1-f'_{l+1,n-m}(1)\left(1-f_{n-m+1,n}(a_j)\right)\right]\vee0.$$
	Combining the last inequality with the fact that $f_{1,l}$ is increasing, we thus obtain
	$$q_n(j)=1-f_{1,n}(a_j)=1-f_{1,l}\circ f_{l+1,n}(a_j)\leq r^{(l,m)}_n(j),$$
	which implies
	\begin{equation}\label{ine1}
		r^{(l,m)}_n(j)-q_n(j)\geq0.
	\end{equation}
	Moreover, by convexity of $f_{1,l}$, we have
	\begin{align}
		r^{(l,m)}_n(j)-q_n(j)&=f_{1,l}\circ f_{l+1,n}(a_j)-f_{1,l}\left(\left[1-f'_{l+1,n-m}(1)\left(1-f_{n-m+1,n}(a_j)\right)\right]\vee0\right)\notag\\
		&\leq f'_{1,l}(1)\,\left\{f_{l+1,n}(a_j)-\left[1-f'_{l+1,n-m}(1)\left(1-f_{n-m+1,n}(a_j)\right)\right]\vee0\right\}\label{eq3}. 
	\end{align} 
	Using the fact that for any $(a,b)\in\R^2$, $a-b\vee0\leq a-b$, we have 
	\begin{align*}
		&f_{l+1,n}(a_j)-\left[1-f'_{l+1,n-m}(1)\left(1-f_{n-m+1,n}(a_j)\right)\right]\vee0\notag\\
		\leq &f_{l+1,n}(a_j)-\left[1-f'_{l+1,n-m}(1)\left(1-f_{n-m+1,n}(a_j)\right)\right]\notag\\
		=&f'_{l+1,n-m}(1)\,q_{n-m,n}(a_j)-q_{l,n}(a_j).
	\end{align*} 
	Combining the last inequality with Inequality (\ref{eq3}), we get
	\begin{align}
		r^{(l,m)}_n(j)-q_n(j)&\leq f'_{1,l}(1)\,\left[f'_{l+1,n-m}(1)\,q_{n-m,n}(a_j)-q_{l,n}(a_j)\right]\notag\\
		&=e^{S_{n-m}}\,q_{n-m,n}(a_j)-e^{S_{l}}\,q_{l,n}(a_j)\notag\\
		&=e^{S_{n-m}}\,q_{n-m,n}(a_j)\,e^{S_{l}}\,q_{l,n}(a_j)\left[e^{-S_{l}}\,q^{-1}_{l,n}(a_j)-e^{-S_{n-m}}\,q^{-1}_{n-m,n}(a_j)\right]\label{eq4}.
	\end{align} 
	Applying (\ref{qn-inverse}) in Lemma \ref{Lem-exp}, we have
	$$q_{n-m,n}(a_j)\leq(1-a_j)\,e^{S_n-S_{n-m}}\quad\text{and}\quad q_{l,n}(a_j)\leq(1-a_j)\,e^{S_n-S_l};$$
	so (\ref{eq4}) yields
	$$r^{(l,m)}_n(j)-q_n(j)\leq (1-a_j)^2\,e^{2S_n}\left[e^{-S_{l}}\,q^{-1}_{l,n}(a_j)-e^{-S_{n-m}}\,q^{-1}_{n-m,n}(a_j)\right].$$
	In addition, by the definitions of $r^{(l,m)}_n(j)$ and $q_n(j)$, we have $r^{(l,m)}_n(j)-q_n(j)\leq1$. Therefore, it holds
	\begin{equation}\label{ine2}
		r^{(l,m)}_n(j)-q_n(j)\leq 1\wedge (1-a_j)^2\,e^{2S_n}\left[e^{-S_{l}}\,q^{-1}_{l,n}(a_j)-e^{-S_{n-m}}\,q^{-1}_{n-m,n}(a_j)\right].
	\end{equation}
	
	Denote
	$$I_3:=n^{3/2}\,\w{\E}_{\l,i}\left[\left|q^2_n(a_j)-\left(r_n^{(l,m)(j)}\right)^2\right|e^{-(\l+1)S_n};\,\tau_y>n\right].$$
	Taking into account (\ref{ine1}), (\ref{range-rnlm}) and (\ref{qn-range}), we have
	\begin{align*}
		\left|q^2_n(a_j)-\left(r_n^{(l,m)(j)}\right)^2\right|&=\left(r_n^{(l,m)}(j)-q_n(j)\right)\left(q_n(j)+r_n^{(l,m)(j)}\right)\notag\\
		&\leq2\left( r_n^{(l,m)}(j)-q_n(j)\right).
	\end{align*}
	Therefore, we have for any $N\in\N$,
	\begin{align*}
		I_3\leq 2n^{3/2}\,\w{\E}_{\l,i}\left[\left(r_n^{(l,m)}(j)-q_n(j)\right)e^{-(\l+1)S_n};\,y+S_n\leq N,\,\tau_y>n\right]+\\
		2n^{3/2}\,\w{\E}_{\l,i}\left[\left(r_n^{(l,m)}(j)-q_n(j)\right)e^{-(\l+1)S_n};\,y+S_n> N,\,\tau_y>n\right].
	\end{align*}
	Using Inequality (\ref{ine2}), we obtain for any $N\in\N$,
	\begin{align}
		I_3&\leq2(1-a_j)^2\,e^{2(N-y)}\,n^{3/2}\,\w{\E}_{\l,i}\left[\left(e^{-S_{l}}\,q^{-1}_{l,n}(a_j)-e^{-S_{n-m}}\,q^{-1}_{n-m,n}(a_j)\right)\,e^{-(\l+1)S_n};\,\tau_y>n\right]+\notag\\
		&\qquad\qquad\qquad\qquad\qquad\qquad\qquad 2n^{3/2}\,\w{\E}_{\l,i}\left[e^{-(\l+1)S_n};\,y+S_n>N,\,\tau_y>n\right]\notag\\
		&=:I_4+I_5.\label{ine3}
	\end{align}
	Let's deal with $I_4$ and $I_5$, respectively. By (\ref{eq-weakly}), we have
	$$I_4=2(1-a_j)^2\,e^{2(N-y)}\,I_2,$$
	where $I_2$ is defined in (\ref{def-I2}).
	Applying Lemma \ref{lem2-2}, we have
	\begin{equation}\label{ine4}
		\lim_{l,m\rightarrow+\infty}\limsup_{n\rightarrow+\infty}I_4=0.
	\end{equation}
	For $I_5$, we have
	\begin{align}
		I_5&=2\,n^{3/2}\,\sum_{p=N}^{+\infty} e^{(\l+1)y} \,\w{\E}_{\l,i}\left[e^{-(\l+1)(y+S_n)};\,y+S_n\in(p,p+1],\,\tau_y>n\right]\notag\\
		&\leq 2n^{3/2}\,e^{(\l+1)y}\sum_{p=N}^{+\infty}e^{-(\l+1)p}\,\w{\P}_{\l,i}\left[y+S_n\in(p,p+1],\,\tau_y>n\right]\label{ine5}.
	\end{align}
	Applying Statement 2 of Proposition 2.9 in \cite{GLL19}, we have
	\begin{equation*}
		n^{3/2}\,\w{\P}_{\l,i}\left[y+S_n\in(p,p+1],\,\tau_y>n\right]\leq c\, \left(1+y\vee0\right)(p+1).
	\end{equation*}
	Combining the last inequality with (\ref{ine5}), we have
	$$I_5\leq 2\,c\,e^{(\l+1)y}\left(1+y\vee0\right)\sum_{p=N}^{+\infty}(p+1)\,e^{-(\l+1)p}.$$ 
	Taking into account the last inequality,  (\ref{ine3}) and (\ref{ine4}), we obtain
	$$\lim_{l,m\rightarrow+\infty}\limsup_{n\rightarrow+\infty}I_3\leq 2\,c\,e^{(\l+1)y}\left(1+y\vee0\right)\sum_{p=N}^{+\infty}(p+1)\,e^{-(\l+1)p}.$$
	Taking the limit as $N\rightarrow+\infty$, we obtain the desired result. \hfill$\square$
\end{pf}

For any $(l,m)\in\N^2$ and $(j,u)\in(\X,\R)$, consider the random variables
$$r_\infty^{(l,m)}(j,u):=1-f_{1,l}\left(\left[1-e^{-S_l}\,e^u\,q^*_m(j)\right]\vee0\right).$$
By Condition \ref{Con3}, we have 
\begin{equation}\label{range-rlm-infty}
	r_\infty^{(l,m)}(j,u)\in[0,1).
\end{equation}

Recall that $\w{V}_{\l}$ and $\w{V}^*_{\l}$, both defined on $\X\times\R$, represent the harmonic functions of the Markov kernels $\w{\bfP}_\l$ and $\w{\bfP}_\l^*$ respectively. One may refer to \cite{GLL18} for detailed positivity of the harmonic function. For any $j\in\X$, let $z_0(j)\in\R$ be the unique real value such that $(j,z)\in\mbox{Supp}\left(\w{V}_{\l}^*\right)$, for any $z>z_0$ and $(j,z)\notin\mbox{Supp}\left(\w{V}_{\l}^*\right)$, for any $z<z_0$.

For any $(i,y)\in\mbox{Supp}\left(\w{V}_{\l}\right)$ and $(j,z)\in\mbox{Supp}\left(\w{V}_{\l}^*\right)$, let $\P^+_{i,y,j,z}$ and $\E^+_{i,y,j,z}$ be respectively the probability law and its associated expectation defined by
\begin{align}
	&\w{\E}^+_{i,y,j,z}\left[g(X_1,\ldots,X_l,X^*_m,\ldots,X^*_1)\right]\notag\\
	=&{\w{V}_\l(i,y)}^{-1}\,{\w{V}_\l^*(j,z)}^{-1}\,\w{\E}_{\l,i,j}\left[g(X_1,\ldots,X_l,X^*_m,\ldots,X^*_1)\,\w{V}_\l(X_l,y+S_l)\,\w{V}_\l^*(X_m^*,z+S_m^*)\,;\,\tau_y>l,\,\tau^*_z>m\right]\label{def-E+},
\end{align}
for any $(l,m)\in\N^2$ and function $g:\X^{l+m}\rightarrow \C$.

Let
$$\w{\sigma}_{\l}^2:=\w{\bfn}_{\l}\left(\rho^2\right)-\w{\bfn}_{\l}^2\left(\rho\right)+2\sum_{n=1}^{+\infty}\left[\w{\bfn}_{\l}\left(\rho\w{\bfP}^n_{\l}\rho\right)-\w{\bfn}_{\l}^2\left(\rho\right)\right].$$
According to Lemma 10.3 in \cite{GLL20}, we have 
$$\displaystyle0<\w{\sigma}_{\l}^2<+\infty.$$

\begin{lemma}\label{lem2-4}
	Assume that the conditions of Theorem \ref{Thm1} are satisfied. Then for any $(i,y)\in \mbox{Supp}\left(\w{V}_{\l}\right)$, $(l,m)\in\N^2$ and $j\in\X$, 
	\begin{align*}
		&\lim_{n\rightarrow+\infty}n^{3/2}\, \w{\E}_{\l,i}\left[\left(r_n^{(l,m)}\right)^2\,e^{-(\l+1)S_n};\,X_{n+1}=j,\,\tau_y>n\right]\\
		=&\frac{\sqrt{2}\,\w{\bfn}_\l(j)\,e^{(\l+1)y}}{\sqrt{\pi}\,\w{\sigma}_{\l}^3}\int_{z_0(j)\vee0}^{+\infty}e^{-(\l+1)z}\,\w{\E}^+_{i,y,j,z}\left[\left(r_{\infty}^{(l,m)}(j,z-y)\right)^2\,\w{V}_\l(i,y)\,\w{V}^*_\l(j,z)\right]dz.
	\end{align*}
\end{lemma}

\begin{pf}
	Fix $(i,y)\in \mbox{Supp}\left(\w{V}_{\l}\right)$, $(l,m)\in\N^2$ and $j\in\X$. Denote
	$$I_6:=n^{3/2}\, \w{\E}_{\l,i}\left[\left(r_n^{(l,m)}\right)^2\,e^{-(\l+1)S_n};\,X_{n+1}=j,\,\tau_y>n\right].$$
	
	Define a function $g:\X^{l+m}\times\R\rightarrow\R_+$ by
	\begin{align*}
		g(i_1,\ldots,i_l,i_{n-m+1},\ldots,i_n,z)=e^{-(\l+1)(z-y)}\,\1_{\{z\geq0\}}\,\w{\bfP}_{\l}(i_n,j)\times\\
		\left\{1-f_{i_1}\circ\cdots\circ f_{i_l}\left(\left[1-e^{z-y-\sum_{k=1}^l\rho(i_k)-\sum_{k=n-m+1}^n\rho(i_k)}\times\left(1-f_{i_{n-m+1}}\circ\cdots\circ f_{i_n}(a_j)\right)\right]\vee0\right)\right\}^2,
	\end{align*}
	for any $\displaystyle(i_1,\ldots,i_l,i_{n-m+1},\ldots,i_n,z)\in\X^{l+m}\times\R$. Note that on the event $\{\tau_y>n\}$,
	
	$$g(X_1,\cdots,X_l,X_{n-m+1},\ldots,X_n,y+S_n)=e^{-(\l+1)S_n}\w{\bfP}_{\l}(X_n,j)\left[r^{(l,m)}_n(j)\right]^2.$$
	Since
	$$0\leq g(i_1,\ldots,i_l,i_{n-m+1},\ldots,i_n,z)\leq e^{-(\l+1)(z-y)}\,\1_{\{z\geq0\}},$$
	the function $g$ is non-negative. In addition, $g$ satisfies the following properties:
	\begin{itemize}
		\item for any $(i_1,\ldots,i_{l+m})\in\X^{l+m}$, the function $z\mapsto g(i_1,\ldots,i_{l+m},z)$ is continuous,
		\item $\displaystyle\max_{i_1,\ldots i_{l+m}\in\X}\,\sup_{z\geq0}g(i_1,\ldots,i_{l+m},z)(1+z)^{2+\varepsilon}<+\infty$, for some $\varepsilon>0$.
	\end{itemize}
	So we can use Markov property and Proposition 2.10 in \cite{GLL19} to obtain
	\begin{align*}
		\lim_{n\rightarrow+\infty}I_6&=\lim_{n\rightarrow+\infty}n^{3/2}\,\w{\E}_{\l,i}\left[g(X_1,\cdots,X_l,X_{n-m+1},\ldots,X_n,y+S_n);\,\tau_y>n \right]\\
		&=\frac{\sqrt{2}}{\sqrt{\pi}\,\w{\sigma}_{\l}^3}\int_{0}^{+\infty}e^{-(\l+1)(z-y)}\sum_{j'\in\X}\w{\E}_{\l,i,j'}\left[\left(r^{(l,m)}_{\infty}(j,z-y)\right)^2\w{\bfP}_\l(X^*_1,j)\times\right.\\
		&\qquad\qquad\qquad\left.\w{V}_\l(X_l,y+S_l)\,\w{V}^*_\l(X^*_m,z+S^*_m);\,\tau_y>l,\,\tau^*_z>m\right]\w{\bfn}_{\l}(j')\,dz.
	\end{align*}
	Since $\w{\bfn}_\l$ is $\w{\bfP}_\l^*$-invariant, we have
	\begin{align*}
		\lim_{n\rightarrow+\infty}I_6=\frac{\sqrt{2}}{\sqrt{\pi}\,\w{\sigma}_{\l}^3}\int_{0}^{+\infty}e^{-(\l+1)(z-y)}\sum_{j_1\in\X}\w{\bfP}_\l(j_1,j)\,\w{\bfn}_\l(j_1)\,\w{\E}_{\l,i}\left[\left(r^{(l,m)}_{\infty}(j,z-y)\right)^2\w{V}_\l(X_l,y+S_l)\right.\times\\
		\left.\w{V}^*_\l(X^*_m,z+S^*_m);\,\tau_y>l,\,\tau^*_z>m\left.\right|\,X_1^*=j_1\right]\,dz.
	\end{align*}
	From (\ref{eq-transition}), we have
	\begin{align}
		&\lim_{n\rightarrow+\infty}I_6\notag\\
		=&\frac{\sqrt{2}}{\sqrt{\pi}\,\w{\sigma}_{\l}^3}\int_{0}^{+\infty}e^{-(\l+1)(z-y)}\sum_{j_1\in\X}\w{\bfP}^*_\l(j,j_1)\,\w{\bfn}_\l(j)\,\w{\E}_{\l,i}\left[\left(r^{(l,m)}_{\infty}(j,z-y)\right)^2\w{V}_\l(X_l,y+S_l)\right.\times\notag\\
		&\qquad\qquad\qquad\qquad\qquad\qquad\qquad\qquad\qquad\left.\w{V}^*_\l(X^*_m,z+S^*_m);\,\tau_y>l,\,\tau^*_z>m\left.\right|\,X_1^*=j_1\right]\,dz.\notag\\
		=&\frac{\sqrt{2}}{\sqrt{\pi}\,\w{\sigma}_{\l}^3}\int_{0}^{+\infty}e^{-(\l+1)(z-y)}\,\w{\bfn}_{\l}(j)\,\w{\E}_{\l,i,j}\left[\left(r^{(l,m)}_{\infty}(j,z-y)\right)^2\w{V}_\l(X_l,y+S_l)\,\w{V}^*_\l(X^*_m,z+S^*_m);\,\tau_y>l,\,\tau^*_z>m\right]\,dz\label{eq-I6}.
	\end{align}
	By (\ref{range-rlm-infty}), (\ref{harmonic}) and (\ref{harmonic*}), we have for $\displaystyle(j,z)\notin\mbox{Supp}\left(\w{V}^*_{\l}\right)$,
	\begin{align*}
		0&\leq\w{\E}_{\l,i,j}\left[\left(r^{(l,m)}_{\infty}(j,z-y)\right)^2\w{V}_\l(X_l,y+S_l)\,\w{V}^*_\l(X^*_m,z+S^*_m);\,\tau_y>l,\,\tau^*_z>m\right]\\
		&\leq\w{\E}_{\l,i}\left[\w{V}_\l(X_l,y+S_l);\,\tau_y>l\right]\,\w{\E}^*_{\l,j}\left[\w{V}^*_\l(X^*_m,z+S^*_m);\,\tau^*_z>m\right]=\w{V}_\l(i,y)\,\w{V}^*_\l(j,z)=0.
	\end{align*}
	Taking into account (\ref{eq-I6}) and using the change of measure (\ref{def-E+}), we can obtain immediately the desired result.  \hfill$\square$
\end{pf}

By (\ref{q*n-exp}) and (\ref{range-eta*}) in Lemma \ref{Lem-exp*}, we have for a given $j\in\X$, the sequence $\left(q^*_m(j)\right)_{m\geq1}$ is non-increasing. Taking into account Condition \ref{Con3}, we have for any $m\in\N$ and $j\in\X$,
\begin{equation}\label{range-q*-m}
	q^*_m(j)\in(0,1).
\end{equation}
Therefore, using monotone convergence theorem, we have $\displaystyle\lim_{m\rightarrow+\infty}q^*_m(j)=q^*_\infty(j)$ and
\begin{equation}\label{exp-inverse-q*-infty}
	q^*_\infty(j)=\left[\frac{1}{1-a_j}+\sum_{k=1}^{+\infty}\eta^*(k)\,e^{-S^*_k}\right]^{-1},
\end{equation}
with
\begin{equation}\label{range-q*-infty}
	q^*_\infty(j)\in(0,1].
\end{equation}
For any $l\in\N$, $j\in\X$ and $u\in\R$, consider the following random variables
$$r_\infty^{(l,\infty)}(j,u):=1-f_{1,l}\left(\left[1-e^{-S_l}\,e^u\,q^*_\infty(j)\right]\vee0\right).$$

\begin{lemma}\label{lem2-5}
	Assume that the conditions of Theorem \ref{Thm1} are satisfied. Then for any $(i,y)\in\mbox{Supp}\left(\w{V}_{\l}\right)$, $(j,z)\in\mbox{Supp}\left(\w{V}_{\l}^*\right)$, $l\in\N$ and $u\in\R$,
	$$\lim_{m\rightarrow+\infty}\w{\E}^+_{i,y,j,z}\left|\left(r_{\infty}^{(l,m)}(j,u)\right)^2-\left(r_{\infty}^{(l,\infty)}(j,u)\right)^2\right|=0.$$
\end{lemma}

\begin{pf}
	Fix $u\in\R$, $(i,y)\in\mbox{Supp}\left(\w{V}_{\l}\right)$, $(j,z)\in\mbox{Supp}\left(\w{V}_{\l}^*\right)$ and $l\in\N$. Since by (\ref{range-rlm-infty}), we have 
	$$r_\infty^{(l,\infty)}(j,u)\in[0,1].$$
	Consequently,
	\begin{align*}
		\left|\left(r_{\infty}^{(l,m)}(j,u)\right)^2-\left(r_{\infty}^{(l,\infty)}(j,u)\right)^2\right|&=\left[r_{\infty}^{(l,m)}(j,u)+r_{\infty}^{(l,\infty)}(j,u)\right]\,\left|r_{\infty}^{(l,m)}(j,u)-r_{\infty}^{(l,\infty)}(j,u)\right|\\
		&\leq2\left|r_{\infty}^{(l,m)}(j,u)-r_{\infty}^{(l,\infty)}(j,u)\right|.
	\end{align*}
	Since the function $f_{1,l}$ is convex on $(0,1)$, we obtain
	\begin{align*}
		\left|\left(r_{\infty}^{(l,m)}(j,u)\right)^2-\left(r_{\infty}^{(l,\infty)}(j,u)\right)^2\right|\leq2f'_{1,l}(1)\left|\left[\left(1-e^{-S_l}\,e^u\,q^*_m(j)\right)\vee0\right]-\left[\left(1-e^{-S_l}\,e^u\,q^*_\infty(j)\right)\vee0\right]\right|.
	\end{align*}
	Since for any $(a,b)\in\R^2$, it holds $|a\vee0-b\vee0|\leq|a-b|$. Applying this to the last term in the right-hand side of the previous inequality and taking into account (\ref{range-q*-m}) and (\ref{range-q*-infty}), we get
	\begin{align*}
		\left|\left(r_{\infty}^{(l,m)}(j,u)\right)^2-\left(r_{\infty}^{(l,\infty)}(j,u)\right)^2\right|&\leq2\,e^{S_l}\,\left|e^{-S_l}\,e^u\,q^*_m(j)-e^{-S_l}\,e^u\,q^*_\infty(j)\right|\\
		&=2\,e^u\left|q^*_m(j)\,q^*_\infty(j)\right|\,\left|{q^*_m(j)}^{-1}-{q^*_\infty(j)}^{-1}\right|\\
		&\leq2\,e^u\,\left|{q^*_m(j)}^{-1}-{q^*_\infty(j)}^{-1}\right|.
	\end{align*}
	In addition, from (\ref{q*n-exp}) and (\ref{range-eta*}) in Lemma \ref{Lem-exp*} and (\ref{exp-inverse-q*-infty}), we thus obtain
	\begin{equation}\label{ine6}
		\left|\left(r_{\infty}^{(l,m)}(j,u)\right)^2-\left(r_{\infty}^{(l,\infty)}(j,u)\right)^2\right|\leq2\,\eta\,e^u\sum_{k=m+1}^{+\infty}e^{-S^*_k}.
	\end{equation}
	Since from Lemma 2.13 in \cite{GLL19}, we have
	$$\w{\E}^+_{i,y,j,z}\left(\sum_{k=m+1}^{+\infty}e^{-S^*_k}\right)=\w{\E}^+_{j,z}\left(\sum_{k=m+1}^{+\infty}e^{-S^*_k}\right)\leq\frac{c\,(1+z\vee0)\,e^z}{\w{V}^*_\l(j,z)}<+\infty,$$
	where $\w{\P}^+_{j,z}$  is the marginal law of $\w{\P}^+_{i,y,j,z}$ on $\sigma\left(X^*_1,\ldots, X^*_n\right)$; taking into account (\ref{ine6}) and applying Lebesgue dominated convergence theorem, we obtain
	\begin{equation}\label{ine7}
		\lim_{m\rightarrow+\infty}\w{\E}^+_{i,y,j,z}\left|\left(r_{\infty}^{(l,m)}(j,u)\right)^2-\left(r_{\infty}^{(l,\infty)}(j,u)\right)^2\right|\leq2\,\eta\,e^u\lim_{m\rightarrow+\infty}\sum_{k=m+1}^{+\infty}\w{\E}^+_{i,y,j,z}\left(e^{-S^*_k}\right).
	\end{equation}
	Moreover, using Lemma 2.13 in \cite{GLL19}, we have 
	\begin{equation}\label{ine8}
		\sum_{k=m+1}^{+\infty}\w{\E}^+_{i,y,j,z}\left(e^{-S^*_k}\right)=\sum_{k=m+1}^{+\infty}\w{\E}^+_{j,z}\left(e^{-S^*_k}\right)\leq\frac{c\,(1+z\vee0)\,e^z}{\w{V}^*_\l(j,z)}\sum_{k=m+1}^{+\infty} k^{-3/2}\xrightarrow{m\rightarrow+\infty}0.
	\end{equation}
	Combining (\ref{ine7}) and (\ref{ine8}), we finally come to the result of this lemma.  \hfill$\square$
\end{pf}

For any $l\in\N$, $j\in\X$ and $u\in\R$, set
$$s_l(j,u):=\left[1-e^{-S_l}\,e^u\,q^*_{\infty}(j)\right]\vee0.$$
Using Lemma 2.13 in \cite{GLL19} and by (\ref{exp-inverse-q*-infty}) and (\ref{range-eta*}) in Lemma \ref{Lem-exp*}, we have for any $(i,y)\in\mbox{Supp}\left(\w{V}_{\l}\right)$ and $(j,z)\in\mbox{Supp}\left(\w{V}_{\l}^*\right)$,
$$\w{\E}^+_{i,y,j,z}\left[{q^*_\infty(j)}^{-1}\right]\leq\frac{1}{1-a_j}+\eta\,\w{\E}^+_{i,y,j,z}\left(\sum_{k=1}^{+\infty}e^{-S^*_k}\right)<+\infty.$$
So $\displaystyle {q^*_\infty(j)}^{-1}<+\infty\quad\w{\P}^+_{i,y,j,z}$--a.s. Therefore,
\begin{equation}\label{range-sl}
	0\leq s_l(j,u)<1.
\end{equation}
For any $j\in\X$ and $u\in\R$, set
$$r_\infty(j,u):=\left[\frac{e^{-u}}{q^*_\infty(j)}+\sum_{k=0}^{+\infty}\eta(k+1)\,e^{-S_k}\right]^{-1}.$$
Using (\ref{exp-inverse-q*-infty}), (\ref{range-eta}) in Lemma \ref{Lem-exp} and (\ref{range-eta*}) in Lemma \ref{Lem-exp*}, we have
\begin{equation}\label{exp-inverse-r-infty}
	r^{-1}_\infty(j,u)\leq e^{-u}\,\left(\frac{1}{1-a_j}+\eta\sum_{k=1}^{+\infty}e^{-S^*_k}\right)+\eta\sum_{k=0}^{+\infty}e^{-S_k}.
\end{equation}

\begin{lemma}\label{lem2-6}
	Assume that the conditions of Theorem \ref{Thm1} are satisfied. Then for any $(i,y)\in\mbox{Supp}\left(\w{V}_{\l}\right)$, $(j,z)\in\mbox{Supp}\left(\w{V}^*_{\l}\right)$ and $u\in\R$,
	$$\lim_{l\rightarrow+\infty}\w{\E}^+_{i,y,j,z}\left|\left(r_\infty^{(l,\infty)}(j,u)\right)^2-r^2_\infty(j,u)\right|=0.$$
\end{lemma}

\begin{pf}
	Fix $(i,y)\in\mbox{Supp}\left(\w{V}_{\l}\right)$, $(j,z)\in\mbox{Supp}\left(\w{V}^*_{\l}\right)$ and $u\in\R$. Denote 
	$$I_7:=\w{\E}^+_{i,y,j,z}\left|\left(r_\infty^{(l,\infty)}(j,u)\right)^{-1}-r^{-1}_\infty(j,u)\right|.$$
	
	Firstly, we will prove that 
	\begin{equation}\label{eq-I7}
		\lim_{l\rightarrow+\infty}I_7=0.
	\end{equation} 
	Using (\ref{qn-inverse}) in Lemma \ref{Lem-exp}, we have
	\begin{equation*}
		\left(r_\infty^{(l,\infty)}(j,u)\right)^{-1}=\left[1-f_{1,l}\left(s_l(j,u)\right)\right]^{-1}=q^{-1}_l\left(s_l(j,u)\right)=\frac{e^{-S_l}}{1-s_l(j,u)}+\sum_{k=0}^{l-1}\eta(k+1)\,e^{-S_k}.
	\end{equation*}
	Therefore, by (\ref{range-eta}) in Lemma \ref{Lem-exp}, we have
	\begin{equation}\label{ine9}
		I_7\leq\w{\E}^+_{i,y,j,z}\left|\frac{e^{-u}}{q^*_\infty(j)}-\frac{e^{-S_l}}{1-s_l(j,u)}\right|+\w{\E}^+_{i,y}\left(\eta\sum_{k=l}^{+\infty}e^{-S_k}\right),
	\end{equation}
	where $\w{\P}^+_{i,y}$  is the marginal law of $\w{\P}^+_{i,y,j,z}$ on $\sigma\left(X_1,\ldots, X_n\right)$. Let's deal with the two terms in the right-hand side of (\ref{ine9}) respectively. For the first term, by (\ref{range-sl}), we have the occurrence of the event $\displaystyle\{e^{-S_l}>e^{-u}/q^*_\infty(j)\}$ implies $\displaystyle\{e^{-S_l}/1-s_l(j,u)>e^{-u}/q^*_\infty(j)\}$. As a result,
	\begin{align}
		\w{\E}^+_{i,y,j,z}\left|\frac{e^{-u}}{q^*_\infty(j)}-\frac{e^{-S_l}}{1-s_l(j,u)}\right|&\leq\w{\E}^+_{i,y,j,z}\left(\left|\frac{e^{-u}}{q^*_\infty(j)}-\frac{e^{-S_l}}{1-s_l(j,u)}\right|\,;\,e^{-S_l}>\frac{e^{-u}}{q^*_\infty(j)}\right)\notag\\
		&\leq\w{\E}^+_{i,y,j,z}\left(\frac{e^{-S_l}}{1-s_l(j,u)}-\frac{e^{-u}}{q^*_\infty(j)}\,;\,e^{-S_l}>\frac{e^{-u}}{q^*_\infty(j)}\right)\notag\\
		&\leq\w{\E}^+_{i,y}\left(\frac{e^{-S_l}}{1-s_l(j,u)}\right)\label{ine10}.
	\end{align} 
	Using Lemma 2.13 in \cite{GLL19}, we have
	\begin{equation}\label{ine11}
		\w{\E}^+_{i,y}\left(\frac{e^{-S_l}}{1-s_l(j,u)}\right)\leq\frac{c\,(1+y\vee0)\,e^y}{l^{3/2}\,\w{V}_\l(i,y)}\,\w{\E}^+_{i,y}\left[\left(1-s_l(j,u)\right)^{-1}\right].
	\end{equation}
	By (\ref{range-sl}), we have $\w{\E}^+_{i,y}\left[\left(1-s_l(j,u)\right)^{-1}\right]<+\infty$. For the second term, using Lemma 2.13 in \cite{GLL19} and Lebesgue dominated convergence theorem,
	\begin{equation}\label{ine12}
		\w{\E}^+_{i,y}\left(\eta\sum_{k=l}^{+\infty}e^{-S_k}\right)\leq\frac{c\,\eta\,(1+y\vee0)\,e^y}{\w{V}_\l(i,y)}\sum_{k=l}^{+\infty}k^{-3/2}.
	\end{equation}
	Combining Inequalities (\ref{ine9}), (\ref{ine10}), (\ref{ine11}) and (\ref{ine12}), we thus obtain
	\begin{equation*}
		I_7\leq\frac{c\,(1+y\vee0)\,e^y}{\w{V}_\l(i,y)}\left(l^{-3/2}+\sum_{k=l}^{+\infty}k^{-3/2}\right)\xrightarrow{l\rightarrow+\infty}0.
	\end{equation*}
	Therefore, (\ref{eq-I7}) is proved.
	
	Secondly, it is evident that for any $l\in\N$, 
	\begin{equation}\label{range-rl-infty-infty}
		r^{(l,\infty)}_\infty(j,u)\in[0,1].
	\end{equation} 
	Since by (\ref{eq-I7}), $\left(r_{\infty}^{(l,m)}(j,u)\right)^{-1}$ converges $\w{\P}^+_{i,y,j,z}$--a.s. to $\left(r_{\infty}^{(l,\infty)}(j,u)\right)^{-1}$, in the same way as the proof of (\ref{eq5}); we can obtain
	$$r_\infty(j,u)\leq1.$$
	Therefore, using the last inequality and (\ref{range-rl-infty-infty}), we get
	\begin{align*}
		\w{\E}^+_{i,y,j,z}\left|\left(r_{\infty}^{(l,m)}(j,u)\right)^2-r^2_{\infty}(j,u)\right|&=\w{\E}^+_{i,y,j,z}\left[\left(r_{\infty}^{(l,m)}(j,u)+r_{\infty}(j,u)\right)\left|r_{\infty}^{(l,m)}(j,u)-r_{\infty}(j,u)\right|\right]\\
		&\leq2\,\w{\E}^+_{i,y,j,z}\left|r_{\infty}^{(l,m)}(j,u)-r_{\infty}(j,u)\right|\\
		&=2\,\w{\E}^+_{i,y,j,z}\left[\left|r_{\infty}^{(l,m)}(j,u)\,r_{\infty}(j,u)\right|\left|\left(r_{\infty}^{(l,m)}(j,u)\right)^{-1}-r^{-1}_{\infty}(j,u)\right|\right]\\
		&\leq 2\,I_7.
	\end{align*}
	Taking the limit as $l\rightarrow+\infty$ and taking into account (\ref{eq-I7}), we can obtain immediately the result of the lemma. \hfill$\square$
\end{pf}

Now, consider the function $u_1$ on $(\X,\R,\X)$ defined by
$$u_1(i,y,j):=\frac{\sqrt{2}\,\w{\bfn}_\l(j)\,e^{(\l+1)y}}{\sqrt{\pi}\,\w{\sigma}_{\l}^3}\int_{z_0(j)\vee0}^{+\infty}e^{-(\l+1)z}\,\w{\E}^+_{i,y,j,z}\left[r_{\infty}^2(j,z-y)\,\w{V}_\l(i,y)\,\w{V}^*_\l(j,z)\right]dz.$$
Using (\ref{exp-inverse-r-infty}) and Lemma 2.13 in \cite{GLL19}, we have for any $(i,y)\in\mbox{Supp}\left(\w{V}_{\l}\right)$ and $(j,z)\in\mbox{Supp}\left(\w{V}_{\l}^*\right)$,
$$\w{\E}^+_{i,y,j,z}\left[r_\infty^{-1}(j,u)\right]<+\infty.$$
Consequently, $r_\infty(j,u)>0$\quad $\w{\P}^+_{i,y,j,z}$--a.s. and thus for any $(i,y)\in\mbox{Supp}\left(\w{V}_{\l}\right)$ and $j\in\X$,
\begin{equation}\label{positivity-u1}
	u_1(i,y,j)>0.
\end{equation}

\begin{lemma}\label{lem2-7}
	Assume that the conditions of Theorem \ref{Thm1} are satisfied. Then for any $(i,y)\in\mbox{Supp}\left(\w{V}_{\l}\right)$ and $j\in\X$, as $n\rightarrow+\infty$,
	$$\w{\E}_{\l,i}\left[e^{-(\l+1)S_n}q_n^2(j)\,;\,X_{n+1}=j,\,\tau_y>n\right]\sim\frac{u_1(i,y,j)}{(n+1)^{3/2}}.$$
\end{lemma}

\begin{pf}
	Fix $(i,y)\in\mbox{Supp}\left(\w{V}_{\l}\right)$ and $j\in\X$. Let
	$$I_8:=(n+1)^{3/2}\,\w{\E}_{\l,i}\left[e^{-(\l+1)S_n}q_n^2(j)\,;\,X_{n+1}=j,\,\tau_y>n\right].$$
	Using Lemmas \ref{lem2-3} and \ref{lem2-4}, we have
	\begin{align*}
		\lim_{n\rightarrow+\infty}I_8&=\lim_{l,m\rightarrow+\infty}\limsup_{n\rightarrow+\infty}(n+1)^{3/2}\,\w{\E}_{\l,i}\left[\left(r_n^{(l,m)}(j)\right)^2\,e^{-(\l+1)S_n}\,;\,X_{n+1}=j,\,\tau_y>n\right]\\
		&=\lim_{l,m\rightarrow+\infty}\frac{\sqrt{2}\,\w{\bfn}_\l(j)\,e^{(\l+1)y}}{\sqrt{\pi}\,\w{\sigma}_{\l}^3}\int_{z_0(j)\vee0}^{+\infty}e^{-(\l+1)z}\,\w{\E}^+_{i,y,j,z}\left[\left(r_{\infty}^{(l,m)}(j,z-y)\right)^2\,\w{V}_\l(i,y)\,\w{V}^*_\l(j,z)\right]dz.
	\end{align*}
	Using (\ref{range-rlm-infty}), Lemmas \ref{lem2-5}, \ref{lem2-6} and Lebesgue dominated convergence theorem, we can obtain
	$$\lim_{n\rightarrow+\infty}I_8=u_1(i,y,j).$$
	\hfill$\square$
\end{pf}

Applying Lemma \ref{lem2-7}, we can obtain Theorem \ref{Thm3}. 
\begin{pot4}
	Fix $(i,j)\in\X^2$. For any $y\in\R$ and $n\in\N$, denote
	$$I_9:=(n+1)^{3/2}\,\w{\E}_{\l,i}\left[q_n^2(j)\,e^{-(\l+1)S_n}\,;\,X_{n+1}=j\right]$$
	and
	$$I_{10}:=(n+1)^{3/2}\,\w{\E}_{\l,i}\left[q_n^2(j)\,e^{-(\l+1)S_n}\,;\,\tau_y\leq n\right].$$
	Then we have
	\begin{align}
		0&\leq I_9-(n+1)^{3/2}\,\w{\E}_{\l,i}\left[q_n^2(j)\,e^{-(\l+1)S_n}\,;\,X_{n+1}=j,\,\tau_y> n\right]\notag\\
		&\leq I_{10}\label{ine13}.
	\end{align} 
	Moreover, by (\ref{qn-range}), we have
	\begin{equation}\label{ine14}
		q^2_n(j)<q_n(j)<q_n.
	\end{equation} 
	Using the fact that the sequence $(q_n)_{n\geq1}$ is non-increasing and (\ref{qn-exp}) in Lemma \ref{Lem-exp}, we can obtain
	$$q_n=\min_{1\leq m\leq n}q_m\leq e^{\min_{1\leq m\leq n}S_m}.$$ 
	Combining this with (\ref{ine14}), we obtain
	\begin{align*}
		I_{10}&\leq n^{3/2}\w{\E}_{\l,i}\left[e^{-(\l+1)S_n}\, e^{\min_{1\leq m\leq n}S_m}\,;\,\tau_y\leq n\right]\\
		&\leq n^{3/2}\,e^{-y}\,\sum_{p=0}^{+\infty}e^{-p}\,\w{\E}_{\l,i}\left[e^{-(\l+1)S_n}\,;\,\min_{1\leq m\leq n}y+S_m\in\left(\right.-(p+1),-p\left.\right],\,\tau_y\leq n\right]\\
		&\leq n^{3/2}\,e^{-y}\,\sum_{p=0}^{+\infty}e^{-p}\,\w{\E}_{\l,i}\left[e^{-(\l+1)S_n}\,;\,\tau_{y+p+1}>n\right]\\
		&=n^{3/2}\,e^{-y}\,\sum_{p=0}^{+\infty}e^{-p}\,\sum_{l=0}^{+\infty}\w{\E}_{\l,i}\left[e^{-(\l+1)S_n}\,;\,y+p+1+S_n\in(l,\,l+1],\,\tau_{y+p+1}>n\right]\\
		&\leq n^{3/2}\,e^{-y}\,\sum_{p=0}^{+\infty}e^{-p}\,e^{(\l+1)(y+p+1)}\sum_{l=0}^{+\infty}e^{-(\l+1)l}\,\w{\P}_{\l,i}\left(y+p+1+S_n\in(l,\,l+1],\,\tau_{y+p+1}>n\right).
	\end{align*}
	Using Statement 2 of Proposition 2.9 in \cite{GLL19} and taking into account $\l\in(-1,0)$, we have
	\begin{align*}
		I_{10}&\leq c\,e^{\l y}\sum_{p=0}^{+\infty}e^{\l p}\,[1+(y+p+1)\vee0]\sum_{l=0}^{+\infty}e^{-(\l+1)l}\,(1+l)\\
		&\leq c\,e^{\l y}\,(1+y\vee0).
	\end{align*}
	Moreover, there exists $y_0(i)\in\R$ such that for any $y\geq y_0(i)$, $(i,y)\in\mbox{Supp}\left(\w{V}_{\l}\right)$. Therefore, using Inequality (\ref{ine13}) and Lemma \ref{lem2-7}, we obtain for any $y\geq y_0(i)$,
	\begin{equation}\label{ine15}
		u_1(i,y,j)\leq\liminf_{n\rightarrow+\infty}I_9\leq\limsup_{n\rightarrow+\infty}I_9\leq u_1(i,y,j)+c\,e^{\l y}(1+y\vee0).
	\end{equation}
	This shows that $\limsup_{n\rightarrow+\infty}I_9$ is finite, which does not depend on $y$ and so $y\mapsto u_1(i,y,j)$ is bounded function. Moreover, by Lemma \ref{lem2-7},
	$$u_1(i,y,j)=\lim_{n\rightarrow+\infty}(n+1)^{3/2}\,\w{\E}_{\l,i}\left[e^{-(\l+1)S_n}q_n^2(j)\,;\,X_{n+1}=j,\,\tau_y>n\right];$$
	and so $y\mapsto u_1(i,y,j)$ is non-decreasing. Let 
	$$u_2(i,j):=\lim_{y\rightarrow+\infty}u_1(i,y,j).$$
	By (\ref{positivity-u1}), we have for any $y\geq y_0(i)$,
	$$u_2(i,j)\geq u_1(i,y,j)>0.$$
	Taking the limit as $y\rightarrow+\infty$ in (\ref{ine15}), we have 
	\begin{equation}\label{cv-I9}
		\lim_{n\rightarrow+\infty} I_9=u_2(i,j).
	\end{equation}
	Let 
	\begin{equation}\label{def-u}
		u(i,j):=\frac{\nu_\l(i)\,u_2(i,j)}{m_j^{\l+1}\nu_\l(j)}>0.
	\end{equation}
	Applying Lemma \ref{Lem-ann}, with $z=1$, we have the following formula:
	\begin{equation}\label{exp-annealed}
		\P_i(Z_{n+1}=1,X_{n+1}=j)=\frac{k^{n+1}(\l)\,\nu_\l(i)}{m_j^{\l+1}\,\nu_\l(j)}\,\w{\E}_{\l,i}\left[e^{-(\l+1)S_n}\,q_n^2(j)\,;\,X_{n+1}=j\right].
	\end{equation}
	Finally, taking into account (\ref{exp-annealed}), (\ref{cv-I9}) and (\ref{def-u}), we obtain
	$$\lim_{n\rightarrow+\infty}\frac{(n+1)^{3/2}}{k^{n+1}(\l)}\,\P_i\left(Z_{n+1}=1,\,X_{n+1}=j\right)=u(i,j)>0.$$
	\hfill$\square$
\end{pot4}






\end{document}